\documentclass{amsart}
\usepackage[T1]{fontenc}

\usepackage{lmodern}

\usepackage{amssymb,amsmath,amsfonts,amsthm,epsfig,amscd,graphicx,tikz-cd}
\usepackage{stmaryrd,mathdots}
\usepackage{IEEEtrantools}
\usepackage[all,cmtip,poly]{xy}
\usepackage{svn}
\usepackage{mathrsfs}
 \usepackage[utf8]{inputenc}

 \usepackage[pdfusetitle]{hyperref}

\hypersetup{
  colorlinks = false,
}

\newtheorem{thm}[subsubsection]{Theorem}
\newtheorem{lemma}[subsubsection]{Lemma}
\newtheorem{lem}[subsubsection]{Lemma}

\newtheorem{cor}[subsubsection]{Corollary}

\newtheorem{prop}[subsubsection]{Proposition}

\newtheorem{atheorem}[subsection]{Theorem}

\newtheorem{defn}[subsubsection]{Definition}

\theoremstyle{remark}
\newtheorem{remark}[subsubsection]{Remark}
\newtheorem{rem}[subsubsection]{Remark}

\numberwithin{equation}{subsection}
\def\nummultline{\addtocounter{subsubsubsection}{1}\begin{multline}}
\def\anumequation{\addtocounter{subsection}{1}\begin{equation}}

\newif\iffinalrun
\iffinalrun
\else
\fi

\iffinalrun
  \newcommand{\need}[1]{}
  \newcommand{\mar}[1]{}
\else
  \newcommand{\need}[1]{{\tiny *** #1}}
  \newcommand{\mar}[1]{\marginpar{\raggedright\tiny #1}}
\fi

\newcommand{\A}{\AA}
\def\C{\CC}
\newcommand{\F}{\FF}

\newcommand{\Q}{\QQ}
\newcommand{\R}{\RR}
\newcommand{\Z}{\ZZ}

\newcommand{\p}{\frakp}

\renewcommand{\AA}{{\mathbb A}}

\newcommand{\CC}{{\mathbb C}}

\newcommand{\FF}{{\mathbb F}}
\newcommand{\GG}{{\mathbb G}}

\newcommand{\QQ}{{\mathbb Q}}
\newcommand{\RR}{{\mathbb R}}

\newcommand{\ZZ}{{\mathbb Z}}

\newcommand{\bG}{\ensuremath{\mathbf{G}}}

\renewcommand{\bf}{\ensuremath{\mathbf{f}}}

\newcommand{\cF}{{\mathcal F}}

\newcommand{\cH}{{\mathcal H}}

\newcommand{\cO}{{\mathcal O}}

\newcommand{\cV}{{\mathcal V}}

\newcommand{\cX}{{\mathcal X}}

\newcommand{\cZ}{{\mathcal Z}}

\newcommand{\rB}{{\mathrm{B}}}

\newcommand{\rG}{{\mathrm{G}}}

\newcommand{\rN}{{\mathrm{N}}}

\newcommand{\rP}{{\mathrm{P}}}

\newcommand{\rT}{{\mathrm{T}}}
\newcommand{\rU}{{\mathrm{U}}}

\newcommand{\rW}{{\mathrm{W}}}

\newcommand{\frakp}{\mathfrak{p}}

\newcommand{\frakX}{\mathfrak{X}}

\newcommand{\Fp}{\F_p}

\newcommand{\Zp}{\Z_p}

\newcommand{\Qp}{\Q_p}

\newcommand{\DXG}[1][m]{\partial X^G_{K(m)}}

\DeclareMathOperator{\codim}{codim}

\DeclareMathOperator{\Ext}{Ext}

\DeclareMathOperator{\Map}{Map}
\DeclareMathOperator{\GL}{GL}

\DeclareMathOperator{\GSp}{GSp}

\DeclareMathOperator{\Hom}{Hom}

\DeclareMathOperator{\ord}{ord}

\DeclareMathOperator{\Spa}{Spa}
\DeclareMathOperator{\Spec}{Spec}

\DeclareMathOperator{\tr}{tr}

\newcommand{\HT}{\mathrm{HT}}

\newcommand{\id}{\mathrm{id}}

\newcommand{\et}{\mathrm{\acute{e}t}}

\newcommand{\toisom}{\xrightarrow{\sim}}

\newcommand{\mbf}{\mathbf}

\newcommand{\Gm}{\GG_m}

\newcommand{\Fl}{\mathscr{F}\!\ell}

\DeclareMathOperator{\Spd}{Spd}

\newcommand{\sD}{\mathscr{D}}

\newcommand{\tG}{\widetilde{G}}
\newcommand{\tK}{\widetilde{K}}
\newcommand{\tH}{\widetilde{H}}

\newcommand{\tP}{\widetilde{P}}

\newcommand{\Perf}{\mathrm{Perf}}

\newcommand{\ol}{\overline}
\newcommand{\ul}{\underline}
\newcommand{\wt}{\widetilde}

\newcommand{\sub}{\subseteq}

\newcommand{\oc}{\mathcal{O}_{C}}

\newcommand{\wh}{\widehat}
\newcommand{\ctX}{\widetilde{\mathcal{X}}}
\newcommand{\tX}{\widetilde{X}}
\newcommand{\ocX}{\overline{\mathcal{X}}}
\newcommand{\rFl}{\mathrm{Fl}}
\newcommand{\us}[1]{\underline{|#1|}}

\usepackage{colonequals}
\newcommand{\defeq}{\colonequals}

\newcommand{\suchthat}{\;\ifnum\currentgrouptype=16 \middle\fi|\;}

\renewcommand{\)}{\right)}

\newcommand{\BS}{\mathrm{BS}}
\newcommand{\BM}{\mathrm{BM}}

\begin{document}
\title{Vanishing theorems for Shimura varieties at unipotent level}
\author{Ana Caraiani, Daniel R. Gulotta, Christian Johansson}
\address{Department of
  Mathematics, Imperial College London,
  London SW7 2AZ, UK}
\email{caraiani.ana@gmail.com}
\address{Mathematical Institute, University of Oxford, Oxford OX2 6GG, UK}
\email{Daniel.Gulotta@maths.ox.ac.uk}
\address{Department of Mathematical Sciences, Chalmers University of Technology and the University of Gothenburg,
  SE-412 96, Sweden}
\email{chrjohv@chalmers.se}

\maketitle

\begin{abstract}
We show that the compactly supported cohomology of Shimura varieties of Hodge type of infinite $\Gamma_1(p^\infty)$-level (defined with respect to a Borel subgroup) vanishes above the middle degree, under the assumption that the group of the Shimura datum splits at $p$. This generalizes and strengthens the vanishing result proved in \cite{arizona}. As an application of this vanishing theorem, we prove a result on the codimensions of ordinary completed homology for the same groups, analogous to conjectures of Calegari--Emerton for completed (Borel--Moore) homology. 
\end{abstract}

\section{Introduction}

This paper proves a generalization of the main geometric result of \cite{arizona}, and gives an application to the bounds on the codimensions of ordinary completed cohomology groups for certain Shimura varieties. Along the way we prove results on finite group quotients of adic spaces and diamonds, and a Poincar\'e duality spectral sequence for ordinary completed cohomology, which we consider to be of independent interest. Before giving a brief introduction to our results, we refer the interested reader to the introduction of \cite{arizona} for further context.

\medskip

Fix a prime $p$. We give an overview of the setup, referring to the main text for details. Let $G$ be a connected reductive group over $\Q$ admitting a Shimura datum of Hodge type. Assume that $G$ is split at $p$ and choose a split model over $\Zp$. Choose a Borel subgroup $B$ of $G$ over $\Z_p$ and let $U\sub B$ be its unipotent radical. If $K \sub G(\Zp)$ is a compact open subgroup, we write $X_{K}$ for the complex Shimura variety for $G$ of level $K$ at $p$ and some fixed tame level\footnote{$K^p$ is assumed to be sufficiently small in a way that we make precise in \S \ref{subsec: anticanonical tower}.} $K^p \sub G(\A_f^p)$, viewed as an algebraic variety. We write $X_K(\C)$ for the corresponding complex manifold. We may state our main vanishing theorem as follows:

\begin{atheorem}[Corollary \ref{strongest vanishing theorem}, Remark \ref{remark on main thm}]\label{main introduction}
Let $d$ be the complex dimension of the Shimura varieties for $G$. Let $H \sub U(\Zp)$ be a closed subgroup.
Then
\[
 \varinjlim_{K \supseteq H} H^i_c(X_{K}(\CC), \Z /p^r ) = 0
\]
for all $r\geq 1$ and all $i>d$.
\end{atheorem} 

The cohomology groups here are singular cohomology with compact support. This theorem generalizes \cite[Theorem 1.1.2]{arizona} in two ways. First, that theorem only treats Shimura varieties for quasi-split (general) unitary and symplectic groups over totally real fields --- here we generalize this to all Hodge type Shimura varieties (in both cases assuming the same splitness conditon at $p$). Second, in \cite[Theorem 1.1.2]{arizona} we require the subgroup $H$ to be contained in the $\Zp$-points of the unipotent radical of the Siegel parabolic; this is a stronger assumption than the containment $H\sub U(\Zp)$ in Theorem \ref{main introduction}.

\medskip

The method of proof is a variation of that of \cite[Theorem 1.1.2]{arizona}, and we refer to the introduction of \cite{arizona} for a more elaborate sketch. Choosing an embedding $\C \hookrightarrow C$ into algebraically closed non-archimedean field $C$ we may view base change the $X_K$ to $C$ and then analytify to get rigid analytic varieties $\cX_K$. As in \cite{scholze-galois}, we make use of compactifications $\ocX_K$ of the $\cX_K$ which are closely related to the minimal compactifications. Through a string of comparison theorems, one reduces Theorem \ref{main introduction} to proving $H_{\et}^i(\ocX, j_! (\cO_{\cX}^+/p))^a = 0$ for $i>d$, where 
\[
j : \cX_H \defeq \varprojlim_{K\supseteq H} \cX_{K} \hookrightarrow \ocX_H \defeq \varprojlim_{K \supseteq H} \ocX_{K}
\]
is the inclusion, the inverse limits are taken as diamonds in the sense of \cite{diamonds}, and $-^a$ denotes the corresponding almost module with respect to $\cO_C$ and its maximal ideal. As in \cite{arizona}, $H_{\et}^i(\ocX_H, j_! (\cO_{\cX}^+/p))^a$ is analyzed using the Leray spectral sequence for a descent $\pi$ of the Hodge--Tate period map which goes from $\ocX_H$ to a quotient of a partial flag variety $\Fl_{G,\mu}$ for $G$, and the ``Bruhat'' stratification of $\Fl_{G,\mu}$ into Schubert cells for the action of $B$. 

Apart from the fact that we treat more general Shimura varieties, there are two principal differences between the argument presented here and that of \cite{arizona} that we wish to point out. The first is that we need to adapt the results of \cite[\S 4]{arizona} on the Schubert cells for the Siegel parabolic on $\Fl_{G,\mu}$ to the Schubert cells for $B$. In fact, it turns out that the arguments flow more naturally in this setting. The second is that our analysis of the fibers of $\pi$ uses some new and different techniques. The argument in \cite{arizona} relies heavily on a general result about the existence of invariant rational neighborhoods for profinite group actions on affinoid adic spaces (\cite[Proposition 5.2.1]{arizona}). Here, we instead make use of some new results on quotients of diamonds by finite groups, which we consider to be of independent interest. A corollary is a strengthening of \cite[Theorem 1.4]{hansen} showing that quotients of affinoid perfectoid spaces are affinoid perfectoid; see Theorem \ref{quotients of aff perf spaces}. We discuss the differences between our argument and the argument of \cite{arizona} in more detail in the introduction to \S \ref{sec: vanishing theorem} and in Remark \ref{differences}.

\medskip

We give one application of Theorem \ref{main introduction} in this paper. Hida's theory of the \emph{ordinary projector} and \emph{ordinary automorphic forms} has played a key role in the $p$-adic study of automorphic forms since its introduction in the 1980's. Hida's constructions come in different flavors, with the most general being in terms of the singular cohomology of locally symmetric spaces. It was later realized by Emerton \cite{emerton-ord1} that the ordinary projector is closely related to the right adjoint of the parabolic induction functor in the mod $p$ and $p$-adic representation theory of $p$-adic reductive groups. Moreover, Hida's construction can be recovered\footnote{This statement needs some care to be made precise. Since this statement is only for context and will not be used in the paper, we will not elaborate on it.} by applying this right adjoint to Emerton's \emph{completed cohomology}, which plays a prominent role in the $p$-adic Langlands program at present (see \cite{calegari-emerton,emerton-icm} for surveys). 

In this paper we follow Hida's approach and look at the ordinary ``completed''\footnote{As is now common, we will occasionally use the term ``completed (co)homology'' to refer to more general (co)limits of (co)homology groups of locally symmetric spaces, in the spirit of the constructions in \cite{calegari-emerton}).} (Borel--Moore) homology groups
\[
\tH^{(\BM),\mathrm{ord}}_i \defeq \varprojlim_{K \supseteq U(\Zp)} H^{(\BM)}_i(X_{K}(\C),\Zp)^{\mathrm{ord}},
\]
where on the right hand side $H_i^{(\BM)}$ denotes $i$-th (Borel--Moore) homology and the superscript $-^{\mathrm{ord}}$ denotes the ordinary part (which is the image of the ordinary projector, i.e. the direct summand where certain $U_p$-like operators for $G$ act invertibly\footnote{We refer to \S~\ref{subsec:ordinary parts} for a precise definition.}). These are direct summands of the corresponding completed cohomology groups
\[
\tH^{(\BM)}_i(U(\Zp)) \defeq \varprojlim_{K \supseteq U(\Zp)} H^{(\BM)}_i(X_{K}(\C),\Zp)
\]
at unipotent level. The module $\tH_i(U(\Zp))$ is dual, in an appropriate sense, to the direct limits appearing in Theorem \ref{main introduction} for $H = U(\Zp)$; see~\eqref{eq:precise dual} for the precise statement. If $T \sub B$ is a maximal split torus, then $\tH^{(\BM),\mathrm{ord}}_i$ is a finitely generated module over the Iwasawa algebra $\mathscr{D}(T_0)$, where $T_0 \defeq T(\Zp)$. Our second main result is then the following.

\begin{atheorem}[Theorem \ref{homology vanishing}]\label{codimensions thm intro}
We have the following:
\begin{enumerate}
\item $\tH^{\BM,\mathrm{ord}}_{i}=0$ for $i >d$. In fact more is true: Let $H \subseteq U(\Zp)$ be a closed subgroup.  Then
\[ \varprojlim_{K \supseteq H} H_i^{\BM}(X_{K}(\C),\Zp) = 0 \]
for $i>d$.

\smallskip

\item
We have
\[ \codim_{\mathscr{D}(T_0)} \tH_i^{\ord} \ge d-i \]
for all $0 \le i \le d$.
\end{enumerate}
\end{atheorem} 

The codimension function may be defined as 
\[
\codim_{\mathscr{D}(T_0)}\tH_i^{\ord} = \inf_j \{j \mid \Ext^j_{\mathscr{D}(T_0)}(\tH_i^{\ord}, \mathscr{D}(T_0))\not = 0\};
\]
we discuss it further in \S \ref{subsec: application}. Theorem \ref{codimensions thm intro} is an analogue of an important conjecture of Calegari and Emerton for completed homology and completed Borel--Moore homology \cite[Conjecture 1.5]{calegari-emerton}. The analogue of Theorem \ref{codimensions thm intro} in that setting was proved by Scholze for Shimura varieties of Hodge type \cite[Corollary 4.2.3]{scholze-galois}. The main ingredients in the proof of Theorem \ref{codimensions thm intro} are Theorem \ref{main introduction} and a ``Poincar\'e duality'' spectral sequence relating ordinary completed Borel--Moore homology to ordinary completed homology, which we consider to be of independent interest. Such a result has previously been announced by Emerton, though relying on a different method than ours. 

The conjecture \cite[Conjecture 1.5]{calegari-emerton} is symmetric when swapping the roles of homology and Borel--Moore homology. However, we remark that one will need more care when formulating (conjectural) versions of Theorem \ref{codimensions thm intro} with homology and Borel--Moore homology swapped, or in analogus situations like eigenvarieties. When swapping homology and Borel--Moore homology, one has to factor in contributions from ordinary boundary homology. In particular, we remark that Theorem \ref{main introduction} fails if one replaces compactly supported cohomology with cohomology, and Theorem \ref{codimensions thm intro}(1) should fail if one replaces Borel--Moore homology with homology due to the presence of boundary homology. For eigenvarieties, conjectures on codimensions have been given by Urban \cite[Conjecture 5.7.5]{urban}, with partial results by Hansen and Newton \cite{hansen-thesis} that are somewhat orthogonal to our Theorem \ref{codimensions thm intro}.

\medskip

Let us now give short overview of the paper. Section \ref{sec: prelim} discusses results on quotients of diamonds by (pro)finite groups which are used in the analysis of the fibers of the Hodge--Tate period map. Section \ref{sec: shimura} introduces the Shimura varieties we consider in this paper and proves the perfectoidness results needed for Theorem \ref{main introduction}, and section \ref{sec: vanishing theorem} is devoted to the proof of Theorem \ref{main introduction}. Finally, section \ref{sec: ordinary parts} introduces ordinary completed (Borel--Moore) homology and proves Theorem \ref{codimensions thm intro}.

\subsection*{Acknowledgments} The authors wish to thank David Hansen, whose joint work in progress with one of us (C.J.) has had an impact on the material in \S \ref{sec: ordinary parts}. A.C. and D.G. wish to thank Matthew Emerton for discussions on ordinary parts and in particular for explaining his approach to the Poincar\'e duality spectral sequence. This paper arose as a continuation of \cite{arizona}, which in turn came out of A.C. and C.J.'s project group at the 2017 Arizona Winter School. We wish to thank the organizers of the A.W.S. again for the fantastic experience, and we wish to thank our collaborators from \cite{arizona} and all participants of the project group. We thank the anonymous referee for their many useful comments. 

This project has received funding from the European Research Council (ERC) under the European Union’s Horizon 2020 research and innovation programme
(grant agreement No.\ 804176). A.C.~was supported by a Royal Society University Research Fellowship and D.G.~was supported by Royal Society grant RP\textbackslash{}EA\textbackslash{}180020. 

\section{Preliminaries on diamonds}\label{sec: prelim}

In this section, we will prove some results on diamonds and v-sheaves that will be used later in the paper. We use the notation, terminology and conventions regarding diamonds and v-sheaves of \cite{diamonds}. Unless otherwise specified, v-sheaves are assumed to be on the category $\mathrm{Perf}$ of perfectoid spaces of characteristic $p$.

\subsection{Quotients of diamond spectra and perfectoid spaces by finite groups}
\begin{prop} \label{surjection finite level}
Let $A$ be a complete Tate $\Zp$-algebra with a continuous left action of a finite
group $G$, and let $A^+$ be an open and integrally closed subring of $A$
that is preserved by $G$.
Let $X = \Spd(A,A^+)$, $X_G = \Spd(A^G,A^{+G})$.
Then $X \times \underline{G} \rightrightarrows X$
is a presentation of $X_G$ as a v-sheaf.
\end{prop}
\begin{proof}
We need to show that $X \to X_G$ and
$X \times \underline{G} \to X \times_{X_G} X$ are surjections of v-sheaves.

The diamonds $X$, $X_G$, $X \times \underline{G}$, and $X \times_{X_G} X$ are spatial (in the last case we use \cite[Corollary 11.29]{diamonds}), hence qcqs.
So $X \to X_G$ and
$X \times \underline{G} \to X \times_{X_G} X$ are qc.  So by
\cite[Lemma 12.11]{diamonds}, it suffices to show that
$|X| \to |X_G|$ and
$|X \times \underline{G}| \to |X \times_{X_G} X|$ are surjections.  The former map is surjective by
\cite[Theorem 3.1]{hansen}.

To prove that the latter map is surjective, we use
the characterization of the topological space of a diamond given in
\cite[Proposition 11.13]{diamonds}.
It suffices to show that if $(K,K^+)$ is a perfectoid field and
$\phi_1, \phi_2 \colon (A,A^+) \to (K,K^+)$ have the same restriction
to $(A^G,A^{+G})$, then $\phi_1, \phi_2$ are related by an element of $G$.
Let $Y = \Spec(A)$, $Y_G = \Spec(A^G)$.
Let $y_1$, $y_2$ be points of $Y(K)$ corresponding to $\phi_1$, $\phi_2$.
By \cite[Thm.~V.4.1(iii)]{SGA31},
$Y \times \underline{G} \to Y \times_{Y_G} Y$ is surjective.
It follows that after extending $K$, $y_1$ and $y_2$ become related by an
element of $G$.
But then they must already be related by an element of $G$ over $K$.
So $\phi_1$ and $\phi_2$ are related by an element of $G$.
\end{proof}

Our main motivation for introducing the above result is to prove Proposition \ref{shimura variety quotient}, but we also mention the following generalization of \cite[Theorem 1.4]{hansen}, which will not be used in the rest of the paper. Recall that the category of adic spaces is, by Huber's definition, a full subcategory of a certain category of locally topologically ringed spaces with valuations on the stalks of the structure sheaf. We follow Kedlaya--Liu in calling such spaces \emph{locally v-ringed spaces}; see \cite[Definitions 8.2.1 and 8.2.2]{kedlaya-liu}.

\begin{thm}\label{quotients of aff perf spaces} 
Let $X$ be a perfectoid space with a right action of a finite group $G$.
Suppose that $X$ has a covering by $G$-stable open subspaces of the form
$\Spa(A,A^+)$ with $A$ perfectoid Tate.
Let $X/G$ be the coequalizer of $X \times \underline{G} \rightrightarrows X$ in the
category of locally v-ringed spaces.  Then $X/G$ is perfectoid.
\end{thm}

The theorem is an
immediate consequence of item (\ref{vring}) of the following proposition.

\begin{prop} \label{v sheaf and v ring quotient}
Let $(A,A^+)$ be a perfectoid Tate-Huber pair with a continuous left action of a
finite group $G$.
Let $X = \Spa(A,A^+)$, and let $X_G = \Spa(A^G,A^{+G})$.  Then:
\begin{enumerate}
\item $X \times \underline{G} \rightrightarrows X \to X_G$ represents
$X_G$ as a coequalizer in the category of v-sheaves on $\mathrm{Perfd}$. \label{coeq perfd}
\item $X \times \underline{G} \rightrightarrows X \to X_G$ represents
$X_G$ as a coequalizer in the category of locally v-ringed spaces. \label{vring}
\end{enumerate}
\end{prop}

\begin{proof}
In the statement of the proposition, we are implicitly using
the fact that $A^G$ is perfectoid \cite[Thm.~3.3.25]{kedlaya-liu-2}
and that any perfectoid space may be regarded as v-sheaf on $\mathrm{Perfd}$
since the v-site is subcanonical \cite[Thm.~8.7]{diamonds}. 

By the same
argument as in Proposition \ref{surjection finite level},
$X \to X_G$ and $X \times \underline{G} \to X \times_{X_G} X$ are
v-covers.  This proves item (\ref{coeq perfd}).

By the argument of \cite[Thm.~3.3]{hansen}, to prove item (\ref{vring}), it is
enough to show that if
$U = X_G\left( \frac{T}{s} \right)$ is a rational subset of $X_G$,
then the natural map $A^G \left< \frac{T}{s} \right> \to A\left< \frac{T}{s} \right>^G$ is an isomorphism.

By item (\ref{coeq perfd}),
in the category of v-sheaves on $\mathrm{Perfd}$,
$\Spa\left(A\left< \frac{T}{s} \right>^G, A\left< \frac{T}{s} \right>^{+G}\right)$ is the
coequalizer of $(X \times_{X_G} U) \times \underline{G} \rightrightarrows X \times_{X_G} U$.
In a topos, coequalizers commute with base change, so this v-sheaf is canonically
isomorphic to $X_G \times_{X_G} U = U = \Spa\left(A^G \left<\frac{T}{s} \right>, A^G \left< \frac{T}{s} \right>^+\right)$.
Therefore $A^G \left< \frac{T}{s} \right> \to A\left< \frac{T}{s} \right>^G$ must be an isomorphism.
\end{proof}
\begin{remark}
The above argument seems to indicate that for general $A$,
the map $A^G \left<T/s \right> \to A \left<T/s \right>^G$ must be
``totally inseparable'' (since $v$-sheafifying generally loses information
about nilpotents and totally inseparable field extensions),
but any ``totally inseparable'' extension of perfectoid rings must
be an isomorphism, giving the result.

This observation led us to find a more
direct proof of Proposition \ref{v sheaf and v ring quotient}(\ref{vring}), which we now sketch.
Again we show that $A^G \left< T/s \right> \to A\left<T/s \right>^G$
is an isomorphism.  By the tilting correspondence, it is enough
to consider the case where $A$ has characteristic $p$.
Since $A^G \left< T/s \right> \to A\left<T/s \right>^G$ induces a bijection of adic spectra and $A^G \left<T/s\right>$ is uniform,
the map is injective and $A^G \left< T/s \right>$ has the subspace topology.
Let $p^m$ be the largest power of $p$ dividing the order of $G$.
Any $a \in A\left<T/s \right>^G$ is a limit of elements of $A[1/s]$;
then $\binom{|G|}{p^m} a^{p^m}$ is a limit of the $p^m$th elemetary
symmetric polynomials in the translates of each element.
So $\binom{|G|}{p^m} a^{p^m}$ is a limit of elements of $A^G[1/s]$; hence
$\binom{|G|}{p^m} a^{p^m} \in A^G \left< T/s \right>$.
Since $\binom{|G|}{p^m}$ is not divisible by $p$, $a^{p^m} \in A^G \left< T/s \right>$, implying $a \in A^G \left< T/s \right>$ since $A^G \left<T/s\right>$ is perfect and $A \left< T/s \right>^G$ is reduced.

\end{remark}

\subsection{Inverse limits of surjections of v-sheaves}

This short subsection consists of a single lemma, we will later need. While we state the lemma in its natural generality, in practice we will only need the case when the indexing system $I$ has a cofinal subsystem isomorphic to $(\Z_{\geq 1}, \leq )$.

\begin{lem} \label{finite to infinite}
Suppose we have cofiltered inverse systems of v-sheaves $X_i, Y_i, i \in I$,
with compatible morphisms $X_i \to Y_i$ that are qcqs and surjective.
Let $X \defeq \varprojlim_i X_i$ and $Y \defeq \varprojlim_i Y_i$.  Then
$X \to Y$ is surjective.
\end{lem}
\begin{proof}
It suffices to show that for any qcqs $Z$ with a map $Z \to Y$,
$Z \times_Y X \to Z$ is surjective.
For each $i$, let
$Z_i \defeq Z \times_{Y_i} X_i$; then $Z \times_Y X = \varprojlim_i Z_i$.
For each $i < j$, the map $Z_j \to Z_i$ is qcqs since it is a base
change of $X_j \to X_i \times_{Y_i} Y_j$, and $X_j \to Y_j$ is qcqs and
$X_i \to Y_i$ is qs.
By the argument of \cite[Lemma 12.17]{diamonds}, we can find a diagram
$W_i$ of spatial diamonds with compatible qcqs surjections $W_i \to Z_i$.
Then each $W_i \to Z$ is qcqs and surjective.  By the argument of loc.~cit.,
$\varprojlim_i W_i \to Z$ is surjective.
So $\varprojlim_i Z_i \to Z$ must be surjective.
\end{proof}

\subsection{Subsheaves and quotients}

Recall that if $S$ is any topological space, then $\ul{S}$ denotes the v-sheaf $Z \mapsto \Map_{cts}(|Z|,S)$ on $\Perf$. In the following, subsets of topological spaces will be equipped with the subspace topology.

\begin{lem} \label{qc generalizing subsheaf}
Let $X$ be a spatial v-sheaf, and let $S$ be a qc and generalizing subset of $|X|$.
Then: 
\begin{enumerate}
\item $X \times_{\us{X}} \ul{S} = \varprojlim_{U \supseteq S} U$,
where $U$ runs over qc open sheaves of $X$ containing $S$. 
\item $X \times_{\us{X}} \ul{S}$ is spatial. \label{qc generalizing spatial}
\item If $X$ is a diamond, then $X \times_{\us{X}} \ul{S}$ is a diamond.
\end{enumerate}
\end{lem}
\begin{proof}
The first claim follows from \cite[Prop.~12.9]{diamonds} (since $S = \bigcap |U|$), and then the remaining claims follow immediately from the first claim and
\cite[Lem.~12.17]{diamonds} and \cite[Lem.~11.22]{diamonds}, respectively.
\end{proof}

The following lemma will be key to the arguments of this paper.

\begin{lem} \label{stabilizers}
Let $X$ be a spatial v-sheaf with a right action of a profinite group $G$. Let $X/G$ denote the quotient of $X$ by $G$ in the category of v-sheaves, and let $\pi \colon X \to X/G$ be the quotient map.
Let $S \subset |X|$ be a qc and generalizing subset.  Suppose that the multiplication map
$S \times G \to S \cdot G$ is a bijection.  Then the natural map
$X \times_{\us{X}} \ul{S} \to X/G \times_{\us{X/G}} \ul{\pi(S)}$ is an isomorphism of v-sheaves.
\end{lem}

\begin{proof}
The product $X/G \times_{\us{X/G}} \ul{\pi(S)}$ can be identified with $\left(X \times_{\us{X}} \ul{S \cdot G} \right)/G$.
Since $\left((X \times_{\us{X}} \ul{S} ) \times \underline{G} \right)/G \cong X \times_{\us{X}} \ul{S}$,
it is enough to show that
$\left(X \times_{\us{X}} \ul{S} \right) \times \underline{G} \to
X \times_{\us{X}} (\ul{S \cdot G})$
is an isomorphism, for which it is enough to show that $\ul{S} \times \ul{G} \to \ul{S \cdot G}$ is an isomorphism. This, in turn, reduces to showing that $S \times G \to S \cdot G$ is a homeomorphism. But this map is easily seen to be spectral and generalizing, and it is bijective by assumption. It is then a homeomorphism by \cite[Tag 09XU]{stacks-project}. 
\end{proof}

\section{Shimura varieties}\label{sec: shimura}

In this section we define the Shimura varieties that we will work with and prove the perfectoidness results and results on flag varieties that we will need. 

\subsection{Setup}\label{setup}
We start by setting up some notation and assumptions. We let $(G,\cH)$ be a Shimura datum of Hodge type, and we assume (crucially!) that $G$ is \emph{split} over $\Qp$. Being split over $\Qp$, $G$ has a natural split reductive model over $\Zp$ \cite[Theorem 6.1.16]{conrad}, and we will use the letter $G$ for this model as well, or sometimes $G_{\Zp}$ for emphasis. Let $E=E(G,\cH)$ be the reflex field of $(G,\cH)$. Since $G$ is split over $\Qp$, $E_{\p}=\Qp$ for every prime $\p$ above $p$ (the local reflex field is the localization of the global reflex field).
From now on we fix such a prime $\p$ of $E$ above $p$, or equivalently an embedding $E \hookrightarrow \Qp$. The remainder of this subsection will be devoted to constructing a particualr embedding $(G,\cH) \hookrightarrow (\tG, \wt{\cH})$ into a Siegel Shimura datum $(\tG,\wt{\cH})$ with certain convenient properties. We start by recalling a lemma from \cite{kisin}.

\begin{lem} \label{integral}
Let $G$ be a reductive group over $\Zp$.
Let $V$ be a finite-dimensonal vector space over $\Qp$, and let
$\rho \colon G_{\Qp} \hookrightarrow \GL(V)$ be a closed embedding of algebraic groups.
Then there exists a lattice $\Lambda \subset V$ so that $\rho$ extends to
a map $G \to \GL(\Lambda)$.
\end{lem}
\begin{proof}
This is proved in \cite[Lemma 2.3.1]{kisin}.  The statement of that lemma includes some additional hypotheses when $p=2$, but these are only used to guarantee that the map of integral models is a closed embedding, which we do not need.
\end{proof}

Next we recall a version of Zarhin's trick.

\begin{lem}[{Zarhin's trick, \cite[\S 2]{zarhin}}] \label{zarhin}
Let $V$ be a finite-dimensional symplectic
$\Q$-vector space and let $(\GSp(V),\cH_V)$ be the associated Siegel Shimura datum. Let $\Lambda \subset V$ be a lattice.  Then
there is a symplectic form on $W = V^4 \oplus (V^{\vee})^4$
preserved up to scaling by $\GSp(V)$ that makes the lattice
$\Lambda^4 \oplus (\Lambda^{\vee})^4 \subset V^4 \oplus (V^{\vee})^4$ self-dual. Moreover, the closed embedding $\GSp(V) \hookrightarrow \GSp(W)$ induces a closed embedding $(\GSp(V),\cH_V) \hookrightarrow (\GSp(W),\cH_W)$ of Siegel Shimura data.
\end{lem}

For convenience, we make the following definition.

\begin{defn}
Let $R$ be a principal ideal domain, and
let $\Lambda$ be finite free module over $R$ equipped with
a symplectic form.
Let $\mu \colon (\Gm)_R \to \GSp(\Lambda)$ be a cocharacter.
Recall \cite[Proposition I.4.7.3]{SGA31} that
$\mu$ induces a decomposition $\Lambda = \bigoplus_{n \in \Z} \Lambda_n$,
where $\mu$ acts on $\Lambda_n$ by $\mu(z)(v) = z^n v$.

We say that $\mu$ is \emph{standard} if
$\Lambda_0$ and $\Lambda_1$ are nonzero and all other
$\Lambda_i$ are zero.
\end{defn}
For any cocharacter $\mu \colon (\Gm)_R \to \GSp(\Lambda)$, the composition
of $\mu$ with the similitude factor $\GSp(\Lambda) \to (\Gm)_R$ must be
of the form $z \mapsto z^n$ for some integer $n$.  Then for each $i \in \ZZ$,
the sympletic form pairs
$\Lambda_i$ with $\Lambda_{n-i}$.
If $\mu$ is standard, then the nondegeneracy of the symplectic form
forces $n=1$ and $\Lambda_0$ and $\Lambda_1$ to be maximal isotropic.

\medskip

We now return to our Hodge type Shimura datum $(G,\cH)$ and the split integral model $G_{\Zp}$. Let us fix a choice of Hodge cocharacter $\mu$ for $G$, viewed as a cocharacter over $\Zp$. If $\rho^\prime : (G,\cH) \hookrightarrow (\GSp(V),\cH_V)$ is any closed embedding into a Siegel Shimura datum, then $\rho^\prime \circ \mu_{\Qp}$ is standard (i.e. a Hodge cocharacter for $(\GSp(V),\cH_V)$, over $\Qp$). The following proposition summarizes the extra conditions we would like to put on our embedding and shows that they are possible to achieve.

\begin{prop} \label{parabolics}
Let $(G,\cH)$, $G_{\Zp}$ and $\mu$ be as above. Then there exists a symplectic $\Q$-vector space $W$ and a closed embedding $\rho_{\Q} \colon (G,\cH) \to (\GSp(W),\cH_W)$ satisfying the following conditions: There exists a self-dual
$\Zp$-lattice $\Lambda \subset W \otimes_{\Q} \Qp$ such that $\rho_{\Qp}$ extends to a homomorphism $\rho_{\Zp} \colon G_{\Zp} \to \GSp(\Lambda)$, and $\rho_{\Zp} \circ \mu$ is standard.

Under these conditions, let $P_{\mu}$ and $P_{\wt{\mu}}$ be the parabolic subgroups of $G$ and $\wt{G} \colonequals \GSp(\Lambda)$ corresponding to $\mu$ and $\wt{\mu} \defeq \rho_{\Zp} \circ \mu$, respectively, as defined in \cite[\S 2.1]{caraiani-scholze} (we remark that these are the parabolics opposite to the ones defined by the Hodge filtration).  Then
$G(\Zp) = G(\Qp) \cap \tG(\Zp)$ and
$P_{\mu}(\Zp) = P_{\mu}(\Qp) \cap \tP_{\mu}(\Zp)$.
\end{prop}

\begin{proof}
Choose an arbitrary closed embedding $\rho' \colon (G,\cH) \to (\GSp(V),\cH_V)$ into a Siegel Shimura datum. By Lemma \ref{integral}, we can find a $\Zp$-lattice $\Lambda_{V_{\Qp}} \subset
V \otimes_{\Q} \Qp$ such that $\rho'_{\Qp}$ extends to a map
$G_{\Zp} \to \GL(\Lambda_{V_{\Qp}})$.
There exists a $\Z$-lattice $\Lambda_V \subset V$ such that
$\Lambda_{V_{\Qp}} = \Lambda_V \otimes_{\Z} {\Zp}$. Set $W = V^4 \oplus (V^{\vee})^4$, $\Lambda = \Lambda_{V_{\Qp}}^4 \oplus (\Lambda_{V_{\Qp}}^{\vee})^4$ and $\Lambda_W = \Lambda_{V}^4 \oplus (\Lambda_{V}^{\vee})^4$. Applying Lemma \ref{zarhin}, we can choose a symplectic form on $W$ giving us an embedding of Shimura data $\rho_\Q \colon (G,\cH) \to (\GSp(W), \cH_W)$ such that $\Lambda_W$ is selfdual. Then $\Lambda$ is also self-dual. The composition 
\[
G_{\Zp} \to \GL(\Lambda_{V_{\Qp}}) \to \GL(\Lambda)
\]
maps $G_{\Qp}$ to $\GSp(W)$, so since $G_{\Qp}$ is dense in the reduced scheme $G_{\Zp}$, the image of $G_{\Zp}$ is contained in $\GSp(\Lambda)$. This gives the $\rho_{\Zp}$ in the statement of the theorem. The cocharacter $\wt{\mu} \defeq \rho_{\Zp} \circ \mu$ is then standard since its generic fiber $\rho_{\Qp} \circ \mu_{\Qp}$ is.

\medskip

Now set $\tG \defeq \GSp(\Lambda)$. To prove the final part of the proposition, first note that it is clear that $G(\Zp) \subseteq G(\Qp) \cap \tilde{G}(\Zp)$. Equality must then hold since $G(\Zp)$ is a maximal compact subgroup of $G(\Qp)$. Then we compute
\[ 
P_{\mu}(\Zp) = G(\Zp) \cap P_\mu(\Qp) = \tG(\Zp) \cap P_\mu(\Qp) = \tG(\Zp) \cap G(\Qp) \cap P_{\wt{\mu}}(\Qp) = G(\Qp) \cap P_{\wt{\mu}}(\Zp), 
\]
finishing the proof.
\end{proof}

From now on, we fix an embedding $(G,\cH) \hookrightarrow (\tG,\wt{\cH})$ into a Siegel Shimura datum satisfying the conditions of Proposition \ref{parabolics}. As with $G$, we will use $\tG$, or sometimes $\tG_{\Zp}$ for emphasis, to denote the split model of $\tG$ over $\Zp$ given by Proposition \ref{parabolics}. Composing our fixed $\mu$ with $G_{\Zp} \to \tG_{\Zp}$ gives a Hodge cocharacter $\wt{\mu}$ for $\tG$, which by Proposition \ref{parabolics} is conjugate over $\Zp$ to the standard cocharacter
$z \mapsto \begin{pmatrix} zI & 0 \\ 0 & I \end{pmatrix}$ in $\tG_{\Zp}$. We let $P_\mu$ and $P_{\wt{\mu}}$ be the parabolic subgroups of $G$ and $\tG$, respectively, that are defined in Proposition \ref{parabolics}.

\subsection{The anticanonical tower}\label{subsec: anticanonical tower}

We now start discussing Shimura varieties. Our notation and definitions will be similar to those of \cite{arizona}, so we will occasionally be rather brief. From now on, we fix a complete algebraically closed extension $C$ of $\Qp$. To start with, it will be more convenient to indicate the full level subgroup in our notation but later we will fix the tame level and only specify the level at $p$. For any compact open subgroup $K \sub G(\A_f)$, always assumed to be neat throughout this paper, we let $X_K$ denote the canonical model (defined over the reflex field $E$) of the Shimura variety of $(G,\cH)$ of level $K$. We set
\[
 \cX_K \defeq (X_K \otimes_E C)^{an},
\]
the analytification (as an adic space over $C$) of the base change of $X_K$ to $C$, via our fixed embedding $E \to \Qp$. If $H \sub G(\A_f)$ is an arbitrary compact subgroup, we set
$$ \cX_H \defeq \varprojlim_{K\supseteq H} \cX_K^\lozenge, $$
where $-^\lozenge$ denotes the diamondification functor on analytic adic spaces over $\Zp$  \cite[Definition 15.5]{diamonds}, the inverse limit is taken over all open $K \sub G(\A_f)$ containing $H$, and the inverse limit exists in the category of diamonds and is locally spatial by \cite[Lemma 11.22]{diamonds}. We note that this is a mild abuse of notation when $H$ itself is open; see \cite[Remark 3.2.8(1)]{arizona} for more details. We define Shimura varieties $\tX_{\tK}$ (over $E$, not $\Q$), $\ctX_{\tK}$ and $\ctX_{\tH}$ for $\tG$ completely analogously, whenever $\tK \sub \tG(\A_f)$ is a compact open subgroup, and $\tH \sub \tG(\A_f)$ is a compact subgroup.  

\medskip

Next, we introduce compactifications. Whenever $K \sub \tK$, there is a natural finite \'etale map $X_K \to \tX_{\tK}$, which extends to a finite map $X_K^\ast \to \tX_{\tK}^\ast$ of minimal compactifications. When $K = \tK \cap G(\A_f)$ by \cite[Proposition 1.15]{deligne-tshimura}, the map $X_K \to \tX_{\tK}$ is a closed immersion but the extension $X_K^\ast \to \tX_{\tK}^\ast$ need not be. Following \cite[\S 4]{scholze-galois}, we define the ad hoc compactification $X^\ast_K \to \ol{X}_K$ to be the universal finite map over which all the $X_K^\ast \to \tX_{\tK}^\ast$ factor (for varying $\tK$ satisfying $K = \tK \cap G(\A_f)$); as noted by Scholze $\ol{X}_K$ is the scheme-theoretic image of $X_K^\ast \to \tX_{\tK}^\ast$ for sufficiently small $\tK$. The right action of $G(\A_f)$ on the tower $(X_K)_K$ extends to an action on $(\ol{X}_K)_K$. We may then analytify: Set $\ol{\cX}_K = (\ol{X} \otimes_E C)^{an}$ and, for $H \sub G(\A_f)$ compact, set $\ol{\cX}_H = \varprojlim_{K\supseteq H} \ol{\cX}_K^\lozenge$. The latter is a spatial diamond by \cite[Lemma 11.22]{diamonds}. We also define $\ctX^\ast_{\tK}$ and $\ctX^\ast_{\tH}$ analogously for $\tG$ (we use minimal compactifications here). We collect some facts about these diamonds.

\begin{lemma}\label{facts about compactifications}
Note that any compact subgroup $H \sub G(\A_f)$ is also a compact subgroup of $\tG(\A_f)$.

\begin{enumerate}
\item Let $H_2 \sub H_1$ be neat compact subgroups of $G(\A_f)$, with $H_2$ normal in $H_1$. Then $|\ol{\cX}_{H_1}| = |\ol{\cX}_{H_2}|/(H_1/H_2)$. Moreover, $|\ocX_{H}|$ is the closure of $|\cX_{H}|$ inside $|\ctX_{H}^*|$ and $\ocX_{H} = \ul{|\ocX_{H}|} \times_{\ul{|\ctX^*_{H}|}} \ctX^*_H$, for any neat compact subgroup $H\sub G(\A_f)$.

\item Let $\tH_2 \sub \tH_1$ be neat compact subgroups of $G(\A_f)$, with $\tH_2$ normal in $\tH_1$. Then $|\ctX^*_{\tH_1}| = |\ctX^*_{\tH_2}|/(\tH_1/\tH_2)$.

\end{enumerate}
\end{lemma}

\begin{proof}
Part (2) is \cite[Lemma 3.2.11]{arizona}. For part (1), the first statement is essentially \cite[Proposition 3.2.15]{arizona} and the second statement is essentially \cite[Corollary 3.2.14]{arizona}; in both cases the proofs of the cited results only use an embedding into a Siegel Shimura datum and go through verbatim in our situation.
\end{proof}

\begin{prop} \label{shimura variety quotient}

\begin{enumerate}
\item For any neat compact subgroups $H_2 \subseteq H_1 \subseteq G(\A_f)$ with $H_2$
normal in $H_1$,
$\ol{\cX}_{H_2} \times \underline{H_1/H_2} \rightrightarrows \ol{\cX}_{H_2}$
is a presentation of $\ol{\cX}_{H_1}$ as a v-sheaf.

\medskip

\item 
For any neat compact subgroups $\tH_2 \subseteq \tH_1 \subseteq \tG(\A_f)$ with $\tH_2$
normal in $\tH_1$,
$\ctX_{\tH_2}^* \times \underline{\tH_1/\tH_2} \rightrightarrows \ctX_{\tH_2}^*$
is a presentation of $\ctX_{\tH_1}^*$ as a v-sheaf.
\end{enumerate}

\end{prop}
\begin{proof}
We first prove part (2). Consider the case where $\tH_1$ and $\tH_2$ are open subgroups of $\tG(\Zp)$. Then we can consider $\ctX_{\tH_1}^*$ and $\ctX_{\tH_2}^*$ as adic spaces. By \cite[Lem.~3.2.2]{arizona}, $\ctX_{\tH_2}^* \to \ctX_{\tH_1}^*$
identifies $\ctX_{\tH_1}^*$ with $\ctX_{\tH_2}^*/(\tH_1/\tH_2)$ (here we take the quotient in the category
of locally v-ringed spaces).  This morphism is finite,
so we can cover $\ctX_{\tH_1}^*$ with affinoids whose pullbacks to $\ctX_{\tH_2}^*$
are affinoids.  Then the result follows from Proposition \ref{surjection finite level} and \cite[Thm.~1.2]{hansen}.

\medskip

In the general case, choose an open compact subgroup $\tK_0$ containing $\tH_1$ and let $\{\tK_i \}_{i\geq 0}$ be a shrinking family of open normal subgroups of $\tK_0$ with $\cap_i \tK_i = \{ 1 \}$. For each $i$, we have shown that
$\ctX^*_{\tK_i \tH_2} \times \underline{\tK_i \tH_1 / \tK_i \tH_2} \rightrightarrows \ctX^*_{\tK_i \tH_2}$ is a presentation of $\ctX^*_{\tK_i \tH_1}$ as a v-sheaf.  The diamonds
$\ctX^*_{\tK_i \tH_j}$ are spatial, so by many applications of \cite[Corollary 11.29]{diamonds}, the relevant maps are qcqs.
Applying Lemma \ref{finite to infinite} then finishes the proof of part (2).

\medskip

We now prove part (1). By Lemma \ref{facts about compactifications}, $\ocX_{H_1} = \ul{|\ocX_{H_1}|} \times_{\ul{|\ctX_{H_1}^*|}} \ctX^*_{H_1}$. Since coequalizers commute with fibre products, we conclude from part (2) that 
\[
\ul{|\ocX_{H_1}|} \times_{\ul{|\ctX_{H_1}^*|}} \ctX_{H_2}^* \times \underline{H_1/H_2} \rightrightarrows \ul{|\ocX_{H_1}|} \times_{\ul{|\ctX_{H_1}^*|}} \ctX_{H_2}^*
\]
is a presentation of $\ocX_{H_1}$ as a v-sheaf. It remains to show that $\ul{|\ocX_{H_1}|} \times_{\ul{|\ctX_{H_1}^*|}} \ctX_{H_2}^* = \ocX_{H_2}$, which in turns reduces to showing that $|\ocX_{H_1}| \times_{|\ctX_{H_1}^*|} |\ctX_{H_2}^*| = |\ocX_{H_2}|$. But this follows directly from Lemma \ref{facts about compactifications}.
\end{proof}

We will now start to only indicate the level at $p$ in the notation for our Shimura varieties. Let $\tK^p \sub \tG(\wh{\Z}^p)$ be a neat open subgroup, which we assume\footnote{This (rather mild) assumption is imposed to be able to apply the perfectoidness results from \cite{scholze-galois} later. It is possible to remove this assumption using the results of \cite{hansen}, but we do not go into this. Also recall that the proof of Proposition \ref{parabolics} furnishes us with a split model of $\tG$ over $\Z$, so it makes sense to talk about $\tG(\wh{\Z}^p)$ and its principal congruence subgroups.} to be contained inside the principal congruence subgroup of level $N$ for some $N\geq 3$, $p\nmid N$. The choice of $\tK^p$ is arbitrary but fixed, unless otherwise indicated. If $\tH \sub \tG(\Zp)$ is a closed subgroup, we now write $\ctX^*_{\tH}$ for what was previously denoted by $\ctX^*_{\tK^p \tH}$, and so on. We make similar conventions for the Shimura varieties, with $K^p \sub G(\A_f)$ a compact open subgroup contained in $\tK^p$, and $H \sub G(\Zp)$ a closed subgroup.

\medskip

Recall the parabolic subgroups $P_\mu \sub G$ and $P_{\wt{\mu}} \sub \tG$ defined at the end of \S \ref{setup}. Let $\tK_0(p)$ denote the parahoric subgroup of $\tG(\Zp)$ with respect to the \emph{opposite} parabolic $\ol{P}_{\wt{\mu}}$ of $P_{\wt{\mu}}$. The Shimura variety $\wt{X}_{\tK_0(p)}$ is the moduli space of principally polarized abelian varieties $(A,\lambda)$ together with a $\tK^p$-level structure and a subspace $W \sub A[p]$ which is  Lagrangian with respect to the $\lambda$-Weil pairing. For any $\epsilon \in [0,1/2)$, we let $\ctX_{\tK_0(p)}(\epsilon)_a \sub \ctX_{\tK_0(p)}$ denote the anticanonical locus of level $\tK_0(p)$ and radius of overconvergence $\epsilon$, which is defined in \cite[Theorem 3.2.15(iii)]{scholze-galois}\footnote{Up to a minor difference in level structures; see the proof of Theorem \ref{anticanonical Siegel} for more details.}. We then set 
$$ \ctX_{\tH}(\epsilon)_a \defeq \ctX_{\tK_0(p)}(\epsilon)_a \times_{\ctX_{\tK_0(p)}} \ctX_{\tH} $$
for all closed subgroups $\tH \sub \tK_0(p)$ and similarly for minimal compactifications. Furthermore, whenever $H \sub G(\Zp) \cap \tK_0(p)$ is a closed subgroup, we set
\[
 \cX_{H}(\epsilon)_a \defeq \ctX_{\tK_0(p)}(\epsilon)_a \times_{\ctX_{\tK_0(p)}} \cX_{H} 
 \]
and similarly for the ad hoc compactifications. We then have the following basic perfectoidness results.

\begin{thm}\label{anticanonical Siegel}
If $\tH \sub \ol{P}_{\wt{\mu}}(\Zp)$ is a closed subgroup, then the diamond $\ctX^*_{\tH}(\epsilon)_a$ is affinoid perfectoid, and the boundary $\wt{\cZ}_{\tH}(\epsilon)_a \defeq \ctX^*_{\tH}(\epsilon)_a \setminus \ctX_{\tH}(\epsilon)_a$ is Zariski closed.
\end{thm}

\begin{proof}
This is \cite[Corollary 3.2.17]{arizona}, up to a minor difference in the level structure. In \cite{arizona} the anticanonical locus is defined on the Shimura variety whose level is contained in the parahoric subgroup corresponding to the parabolic $P_{\wt{\mu}}$ (following \cite{scholze-galois}), but $P_{\wt{\mu}}$ and $\ol{P}_{\wt{\mu}}$ are conjugate (by the longest element of the Weyl group) and the anticanonical loci corresponds, so we may conjugate to get the result.
\end{proof}

\begin{cor}\label{anticanonical tower}

If $H \sub \ol{P}_{\mu}(\Zp)$, then the diamond $\ol{\cX}_H(\epsilon)_a$ is affinoid perfectoid, and the boundary $\cZ_H(\epsilon)_a \defeq \ol{\cX}_{H}(\epsilon)_a \setminus \cX_H(\epsilon)_a$ is Zariski closed.
\end{cor}

\begin{proof}
This follows from Theorem \ref{anticanonical Siegel} in exactly the same way as \cite[Theorem 3.2.18]{arizona} follows from \cite[Corollary 3.2.17]{arizona}.
\end{proof}

We remark that at this stage we haven't proved that $\ol{\cX}_H(\epsilon)_a$ is non-empty, but the result and its proof still make sense. In fact we will not need to separately prove the non-emptiness; it follows from Theorem \ref{Noncanonical locus}.

\subsection{Flag varieties and the Hodge--Tate period map}

We begin by briefly recalling some material from \cite[\S 4]{arizona}. Let $\rG$ be a split connected reductive group over $\Qp$ with a split maximal torus $\rT$ and a Borel subgroup $\rB \supseteq \rT$. Let $\Phi=\Phi(\rG,\rT)$ be the roots of $\rG$ with respect to $\rT$, let $\Phi^+ \sub \Phi$ (resp. $\Phi^- \sub \Phi$) denote the positive (resp.~negative) roots with respect to $\rB$, and let $\Delta \sub \Phi^+$ be the simple roots. Let $\rP=\rP_I\supseteq B$ be the standard parabolic corresponding to a subset $I \sub \Delta$, with unipotent radical $\rN=\rN_I$, and let $\Phi_\rP \sub \Phi$ be the root system of the Levi factor of $\rP$ with respect to (the image of) $\rT$. Let $\rW \defeq W(\rG,\rT)$ and $\rW_\rP \defeq W(\rP/\rN,\rT)$ be the respective Weyl groups.

\medskip
Let $\rU \sub \rB$ be the unipotent radical. As before, we use overlines to denote opposites: We have $\ol{\rB}$, the opposite Borel of $\rB$, with unipotent radical $\ol{\rU}$. Recall that $\ol{\rB} = w_0 \rB w_0$, where $w_0 \in W$ is the longest element. We will look at the stratification of the flag variety $\rG/\rP$ into orbits for $\ol{\rB}$. To this end, we recall the generalized Bruhat decomposition
$$ \rG = \bigsqcup_{w \in \rW/\rW_\rP} \rB w \rP $$
from \cite[Corollaire 5.20]{borel-tits}. From this one easily deduces the decomposition
$$ \rG = \bigsqcup_{w \in \rW/\rW_\rP} \ol{\rB} w \rP $$
using that $\ol{\rB} =w_0 \rB w_0$. We note that $\ol{\rB}\rP$ is the ``big cell'', which is open in $\rG$. The following is \cite[Lemma 4.3.1]{arizona}.

\begin{lem}\label{schubert cell inside affine}
For any $w\in \rW$, $\ol{\rB}w\rP \sub w\ol{\rB}\rP$. In particular, we have the open cover
$$ \rG = \bigcup_{w \in \rW/\rW_\rP} w \ol{\rB} \rP. $$
\end{lem}

We now introduce some more notation. We let $\rFl_\rG \defeq \rG/\rP$ be the partial flag variety of parabolics conjugate to $\rP$. For $w\in \rW$, we have affine open subsets $w\ol{\rB}\rP/\rP \sub \rFl_\rG$, whose stabilizer is $\ol{\rP}_w \defeq w\ol{\rP}w^{-1}$ (since $\ol{\rB}\rP = \ol{\rP}\rP$). The following is the analogue of \cite[Lemma 4.3.2]{arizona}; it is key to proving the stronger vanishing theorem in this article.

\begin{lem}\label{dimension count}
We have $ \dim \ol{\rU} - \dim (\ol{\rU} \cap \ol{\rP}_w) + \dim (\ol{\rB}w\rP/\rP) = \dim \ol{\rN}$.
\end{lem}

\begin{proof}
We start by observing that $\dim \ol{\rU} = \#\Phi^-$. Then, note that the second term $\dim (\ol{\rU} \cap \ol{\rP}_w)$ is equal to $\#(\Phi^- \cap w(\Phi^{-}\cup \Phi_{\rP}))$. For the third term, we first observe that
$$ \ol{\rB}w\rP/\rP \cong \ol{\rB}/(\ol{\rB} \cap w\rP w^{-1}) $$
and then the latter has dimension
$$ (\dim \rT + \#\Phi^{-}) - (\dim \rT + \#(\Phi^{-} \cap w(\Phi^{+} \cup \Phi_{\rP}))) = \#\Phi^{-} - \#(\Phi^{-} \cap w(\Phi^{+} \cup \Phi_{\rP})).$$
The left hand side of the equality we want to prove is then 
$$ \#\Phi^- - \#(\Phi^- \cap w(\Phi^{-}\cup \Phi_{\rP})) + \#\Phi^{-} - \#(\Phi^{-} \cap w(\Phi^{+} \cup \Phi_{\rP})), $$
which is equal to
$$ \#(\Phi^- \setminus w(\Phi^{-}\cup \Phi_{\rP})) + \#(\Phi^{-} \setminus w(\Phi^{+} \cup \Phi_{\rP})). $$
But this last line is just
$$ \#(\Phi^{-} \setminus w(\Phi_{\rP})), $$
and this is independent of $w$, since $\alpha\in \Phi_\rP$ if and only if $-\alpha\in \Phi_\rP$, and hence precisely one of $w(\alpha)$ and $w(-\alpha)=-w(\alpha)$ will be negative. In particular, if we set $w=1$, we get $\#(\Phi^{-}\setminus \Phi_\rP)$, which is equal to the dimension of $\ol{N}$ as desired.
\end{proof}

Next, we leave the setting above and discuss the Hodge--Tate period map for our Shimura varieties. Let $\mbf{1} \sub G(\Zp) \sub \tG(\Zp)$ denote the trivial subgroup, and set $\Fl_{G,\mu} \defeq (G/P_\mu \otimes_{\Qp}C)^{an}$ and $\Fl_{\tG,\wt{\mu}} \defeq (\tG/P_{\wt{\mu}} \otimes_{\Qp}C)^{an}$. These flag varieties have natural models over $\Qp$, and there is a natural Zariski closed embedding $\Fl_{G,\mu} \sub \Fl_{\tG,\wt{\mu}}$. We have Hodge--Tate period maps
$$ \pi_\HT : \cX_{\mbf{1}} \to \Fl_{G,\mu},$$
$$ \wt{\pi}_\HT : \ctX_{\mbf{1}}^* \to \Fl_{\tG,\wt{\mu}} $$
constructed in \cite{scholze-galois, caraiani-scholze}. When $K^p \sub \tK^p$, we have a commutative diagram
  \[
    \xymatrix{\ocX_{\mbf{1}} \ar[r]^-{\pi_\HT} \ar[d] & \Fl_{G,\mu}\ar[d] \\ \ctX^*_{\mbf{1}} \ar[r]^-{\wt{\pi}_{\HT}} & \Fl_{\tG,\wt{\mu}}}.
    \]
By~\cite[Theorem 4.1.1]{scholze-galois}, $\ocX_{\mbf{1}}$ and $\ctX^*_{\mbf{1}}$ are perfectoid spaces. The existence of the commutative diagram follows by combining the proof of~\cite[Theorem 2.1.3(i)]{caraiani-scholze} with the proof of~\cite[Theorem 3.3.1]{arizona}. We remark that $\pi_\HT$ is equivariant for the action of $G(\Qp)$, where $\ocX_\mbf{1}$ is given the standard right $G(\Qp)$-action and $\Fl_{G,\mu}$ is given the right $G(\Qp)$-action that is inverse to the standard left action (the analogous remark applies to $\wt{\pi}_\HT$). Next, we define some ``topological'' Hodge--Tate period maps, as in \cite[\S 4.5]{arizona}. By Lemma \ref{facts about compactifications} we may, for any closed subgroup $H\sub G(\Zp)$, define a map
\[
|\pi_H| : |\ocX_H| \to |\Fl_{G,\mu}|/H
\]
by quotienting out the map $|\pi_\HT|$ by $H$. Moreover, we define a morphism
\[
 \pi_H : \left( \ocX_H \right)_\et \to |\Fl_{G,\mu}|/H 
\]
of sites by precomposing $|\pi_H|$ with the natural map $\left( \ocX_H \right)_\et \to |\ocX_H|$.

\subsection{Perfectoid loci}\label{perfectoid loci}

Choose a maximal split torus and a Borel subgroup $T \sub B\sub P_\mu$ of $G$, all over $\Qp$. Let $W$ be the Weyl group of $G$ with respect to $T$ and let $W_\mu \defeq W_{P_\mu}$ in the notation of the previous subsection. We will define two open covers of $\Fl_{G,\mu}$. The first is by Zariski open affine subsets. We set
$$ V_w \defeq (w\ol{B}P_\mu /P_\mu \otimes_{\Qp} C)^{an} $$
for any $w\in W$; this is an open cover by Lemma \ref{schubert cell inside affine} and $V_w$ only depends on the coset $wW_\mu$. The second cover is the analogous cover by open affinoid subsets $\cV_w$, $w \in W/W_\mu$. One way to define $\cV_w$ is as the rigid generic fibre of the formal completion along $p=0$ of the $\cO_C$-scheme
$$ (w\ol{B}_{\Zp}P_{\mu,\Zp} / P_{\mu,\Zp}) \otimes_{\Zp} \cO_C, $$ 
where we have added the subscript $\Zp$ to emphasize that we are consider the models of these algebraic groups over $\Zp$ (all parabolic subgroups of $G$ over $\Qp$ extend uniquely to parabolic subgroups of $G_{\Zp}$). Set 
$$ \gamma \defeq \mu(p) \in G(\Qp). $$
Then, by the definition of $P_\mu$, it follows that
$$ V_1 = \bigcup_{k\geq 0} \cV_1 \gamma^{-k}. $$
Moreover, the open subsets $\cV_1 \gamma^k$, for $k \leq 0$, form a basis of open neighborhoods of the base point in $\Fl_{G,\mu}$. For any closed subgroup $H\sub \ol{P}_\mu(\Zp)$, we set $\ocX_{H,1} \defeq \pi_H^{-1}(|V_1|/H)$; this is a locally spatial diamond. Note that $\ocX_{\mbf{1},1}$ is non-empty since its translates by elements in $W$ cover $\ocX_\mbf{1}$. It follows that $\ocX_{H,1}$ is non-empty as well, for all $H \sub \ol{P}_\mu(\Zp)$. Our next result generalizes \cite[Theorem 4.5.2]{arizona}.

\begin{thm}\label{Noncanonical locus}
For any closed subgroup $H\sub \ol{P}_\mu(\Zp)$, the locally spatial diamond $\ocX_{H,1}$ is a perfectoid space. More precisely, $|\ocX_{H,1}|$ is covered by the increasing union of quasi-compact open subsets $|\ocX_{\mbf{1}}(\epsilon)_a|\gamma^{-k} /H$ for $k\in \Z_{\geq 0}$ (and sufficiently small $\epsilon > 0$), and the corresponding spatial diamonds are affinoid perfectoid with Zariski closed boundary.  
\end{thm}

\begin{proof}
We may identify $|\ocX_{H,1}|$ with $|\ocX_{\mbf{1},1}|/H$. The first step is to show that the $|\ocX_{\mbf{1}}(\epsilon)_a|\gamma^{-k}$ cover $|\ocX_{\mbf{1},1}|$.
By \cite[Proposition 3.3.4]{arizona},
$\left|\widetilde{\cX}^*_1(\epsilon)_a\right| \gamma^{-k}$ cover
$\left|\widetilde{\cX}^*_{\mbf{1},1}\right|$.
Then we just need to observe that
$|\ocX_{\mbf{1}}(\epsilon)_a|\gamma^{-k} = \left|\widetilde{\cX}^*_1(\epsilon)_a\right| \gamma^{-k} \cap |\ocX_{\mbf{1}}|$ and
$|\ocX_{\mbf{1},1}| \subseteq |\widetilde{\cX}^*_{\mbf{1},1}| \cap |\ocX_{\mbf{1}}|$.
After this, it remains to show that the open subdiamonds of $\ocX_H$ given by the open subsets $|\ocX_{\mbf{1}}(\epsilon)_a|\gamma^{-k} /H$ are affinoid perfectoid with Zariski closed boundary. But $\gamma^k$ induces an isomorphism
\[
 \gamma^k : \ocX_H \toisom \ocX_{\gamma^{-k} H \gamma^k} 
 \]
which identifies $|\ocX_{\mbf{1}}(\epsilon)_a|\gamma^{-k} /H$ with $|\ocX_{\gamma^{-k} H \gamma^k}(\epsilon)_a| $, and $\ocX_{\gamma^{-k} H \gamma^k}(\epsilon)_a$ is affinoid perfectoid with Zariski closed boundary by Corollary \ref{anticanonical tower}, since 
\[
\gamma^{-k} H \gamma^k \sub \gamma^{-k} \ol{P}_\mu (\Zp) \gamma^k \sub \ol{P}_\mu (\Zp).
\]
\end{proof}

We now consider the situation for general $w$. In this case, the parabolic $\ol{P}_{\mu,w} \defeq w \ol{P}_\mu w^{-1}$ stabilizes $V_w$. Note that $\ol{P}_{\mu,w}(\Zp) = w \ol{P}_\mu(\Zp) w^{-1}$ since $G$ is split over $\Qp$ and we have chosen the natural split model over $\Zp$, so $w$ has a representative in $G(\Zp)$. For any closed subgroup $H \sub \ol{P}_{\mu,w} (\Zp)$, we may define
\[
 \ocX_{H,w} \defeq \pi_H^{-1}(|V_w|/H). 
 \]
Then these spaces are perfectoid.

\begin{cor}\label{General noncanonical loci}
For any closed subgroup $H\sub \ol{P}_{\mu,w}(\Zp)$, the locally spatial diamond $\ocX_{H,w}$ is a perfectoid space. More precisely, $|\ocX_{H,w}|$ is covered by the increasing union of quasi-compact open subsets $|\ocX_{\mbf{1}}(\epsilon)_a|\gamma^{-k}w^{-1} /H$ for $k\in \Z_{\geq 0}$ (and sufficiently small $\epsilon > 0$), and the corresponding spatial diamonds are affinoid perfectoid with Zariski closed boundary. 
\end{cor}

\begin{proof}
This follows from Theorem \ref{Noncanonical locus} by looking at the commutative diagram
  \[
    \xymatrix{(\ocX_{w^{-1}Hw})_\et \ar[r]^{w^{-1}} \ar[d]^{\pi_{w^{-1}Hw}} & (\ocX_{H})_\et \ar[d]^{\pi_{H}} \\  |\Fl_{G,\mu}|/w^{-1}Hw \ar[r]^{w^{-1}} & |\Fl_{G,\mu}|/H, }
    \]
where the horizontal maps are isomorphisms, since $V_w = V_1 w^{-1}$ so $\cX^*_{w^{-1}Hw,1}w^{-1} = \cX^*_{H,w}$.
\end{proof}

In this paper, we will only use Corollary \ref{General noncanonical loci} in the situation when $H \sub \ol{U}_w(\Zp)$, where $\ol{U}_w = \ol{U}\cap \ol{P}_{\mu,w}$ and we recall that $\ol{U}$ is the unipotent radical of $\ol{B}$.

\section{The vanishing theorem}\label{sec: vanishing theorem}

In this section we prove Theorem \ref{main complex}, our main result. Our arguments follow those of \cite[\S 5]{arizona} closely, with a few differences. As these are somewhat technical, we discuss them in Remark \ref{differences} when all necessary objects have been introduced. For the sake of readability, we have elected to reproduce some arguments that appear in identical form in \cite{arizona}.

\subsection{First reductions}

We start by stating our main theorem. To state it, we define, for $m \in \Z_{\geq 1}$,
$$ K_1(p^m) \defeq \{g\in G(\Zp) \mid (g\text{ mod }p) \in \ol{U}(\Z / p^m ) \}. $$
Note that $\bigcap_{m \geq 1} K_1(p^m) = \ol{U}(\Zp)$. We let $d$ be the dimension of the Shimura varieties for our Shimura datum $(G,\cH)$; we have $d=\dim \ol{N}_\mu$.

\begin{thm}\label{main complex}
Let $K \sub G(\Zp)$ be an open compact subgroup. Then
$$ \varinjlim_m H^i_c(X_{K\cap K_1(p^m)}(\CC), \Z /p^r ) = 0 $$
for all $r\geq 1$ and all $i>d$.
\end{thm}  

Here the cohomology is singular cohomology (with compact supports) of the complex manifold $X_{K\cap K_1(p^m)}(\CC)$. The following more general version follows directly.

\begin{cor}\label{strongest vanishing theorem}
Let $H \subseteq \ol{U}(\Zp)$ be a closed subgroup.  Then 
\[ \varinjlim_{K \supseteq H} H^i_c(X_{K}(\CC), \Z /p^r ) = 0 \]
for all $r\geq 1$ and all $i>d$.
\end{cor}

\begin{rem}\label{remark on main thm}
This is Theorem \ref{main introduction}. In the formulation of Theorem \ref{main introduction}, the $U$ used there was the unipotent radical of an arbitrarily chosen Borel subgroup of $G$ over $\Zp$, whereas Corollary \ref{strongest vanishing theorem} is formulated in terms of a specific choice (depending on the choice of $\mu$). The two formulations are easily seen to be equivalent, since all Borel subgroups over $\Zp$ are conjugate by elements of $G(\Zp)$ (see e.g. \cite[Corollary 5.2.13]{conrad}). Equivalently, it is possible to conjugate the choice of $\mu$ to make $\ol{U}$ arbitrary. 
\end{rem}

In this subsection we make a series of arguments, as in \cite[\S 5.1]{arizona}, to reduce Theorem \ref{main complex} to particular statement in $p$-adic geometry. First, note that by dévissage it suffices to treat the case $r=1$, and that by applying comparison theorems (between singular and étale cohomology of varieties over $\CC$, between étale cohomology of varieties over $C$ and their analytification, and also invariance of algebraically closed field for étale cohomology of varieties), Theorem \ref{main complex} is equivalent to
\[
 \varinjlim_m H^i_{\et}(\ocX_{K\cap K_1(p^m)}, j_! \Fp ) = 0
 \]
for $i > d$. Here and in the rest of this section we write $j$ for any open immersion $ Y \to \ol{Y}$ where $\ol{Y}$ is a locally spatial diamond with a (fixed) map to $\ocX_{G(\Zp)}$, $Y = \ol{Y} \times_{\ocX_{G(\Zp)}} \cX_{G(\Zp)}$ and $ Y \to \ol{Y}$ is the projection onto the first factor. Applying the primitive comparison theorem \cite[Theorem 3.13]{scholze-survey} and \cite[Lemma 5.1.3]{arizona} gives us 
\[
 \varinjlim_m H^i_{\et}(\ocX_{K\cap K_1(p^m)}, j_! \Fp )^a \otimes_{\Fp} \cO_C/p = \varinjlim_m H^i_{\et}(\ocX_{K\cap K_1(p^m)}, j_! (\cO^+_{\cX_{K\cap K_1(p^m)}}/p) )^a, 
\]
where $-^a$ denotes the almost setting with respect to the maximal ideal of $\oc$. Applying the almost version of \cite[Proposition 2.2.1]{arizona} and some results from \cite[\S 2.3]{arizona} (see the paragraph after \cite[Proposition 5.1.4]{arizona} for more details), we get 
\[
 \varinjlim_m H^i_{\et}(\ocX_{K\cap K_1(p^m)}, j_! (\cO^+_{\cX_{K\cap K_1(p^m)}}/p) )^a = H^i_{\et}(\ocX_{K\cap \ol{U}(\Zp)}, j_! (\cO^+_{\cX_{K\cap \ol{U}(\Zp)}}/p) )^a.
 \]
From here on we make the following convention: For any quasi-pro-étale $Y \to \cX_{G(\Zp)}$, we write $\cO_\cX^+/p$ for the étale sheaf $\cO_Y^+/p$; this is somewhat justified by \cite[Lemmas 2.3.2 and 2.3.3]{arizona}.

\medskip

From now on, we set $H_w \defeq K \cap \ol{U}_w(\Zp)$ for any $w \in W/W_\mu$, and we let $H \defeq H_1 = K \cap \ol{U}(\Zp)$ for simplicity. We consider the morphism
$$ \pi_H : (\ol{\cX}_H)_\et \to |\Fl_{G,\mu}|/H $$
and its Leray spectral sequence
$$ E_2^{rs} = H^r(|\Fl_{G,\mu}|/H, R^s\pi_{H,\ast}j_!(\cO_\cX^+/p)^a) \implies H^{r+s}_\et(\ol{\cX}_H, j_!(\cO^+_\cX/p)^a). $$
Define, for $w\in W/W_\mu$, 
$$ \Fl_{G,\mu}^w \defeq ( \ol{B}wP_\mu/P_\mu \otimes_{\Qp} C)^{an}. $$
These are generalized Schubert cells, and they form a Zariski stratification of $\Fl_{G,\mu}$. Note also that they are stable under $H$. The following is the key result.

\begin{thm}\label{key}
Let $w \in W/W_\mu$ and let $x\in |\Fl_{G,\mu}^w|/H$. Then
$$ \left( R^i\pi_{H,\ast}j_!(\cO_\cX^+/p)^a \right)_x = 0 $$
for $i > d - \dim \Fl^w_{G,\mu}$.
\end{thm}

\begin{proof}[Proof of Theorem \ref{main complex}]
By the Leray spectral sequence for $\pi_H$ it suffices to prove that $H^r(|\Fl_{G,\mu}|/H, R^s\pi_{H,\ast}j_!(\cO_\cX^+/p)^a)=0$ for $r+s > d$. Fix $r$ and assume that $s > d-r$. Let $S_r$ be the set of $w\in W/W_\mu$ for which $\dim \Fl_{G,\mu}^w < r$, set $Y_r \defeq \bigcup_{w \in S_r} \Fl_{G,\mu}^w$ and let $\ol{Y}_r$ be the Zariski closure of $Y_r$. Since $Y_r$ is $\ol{U}(\Zp)$-invariant and has dimension $<r$, the same holds for $\ol{Y}_r$. By \cite[Lemma 3.2.3]{extensions}, $|\ol{Y}_r|/H$ is a spectral space of dimension $<r$. 
Set $\cF^s \defeq R^s\pi_{H,\ast}j_!(\cO_\cX^+/p)^a$; we claim that it is supported on $|\ol{Y}_r|/H$. Take $x \notin |\ol{Y}_r|/H$; by construction $x \notin |\Fl^w_{G,\mu}|/H$ for all $w\in S_r$. It follows that $x \in |\Fl^{w^\prime}_{G,\mu}|/H$ for some $w^\prime$ such that $\dim \Fl^{w^\prime}_{G,\mu} \geq r$. Hence $\cF_x^s = 0$ by Theorem \ref{key} since $s > d - r \geq d - \dim \Fl^{w^\prime}_{G,\mu}$, so $\cF^s$ is indeed supported on $|\ol{Y}_r|/H$. By \cite[Corollary 4.6]{scheiderer}, 
\[
H^r(|\Fl_{G,\mu}|/H, \cF^s) = H^r(|\ol{Y}_r|/H, \cF^s) = 0,
\]
since $|\ol{Y}_r|/H$ is a spectral space of dimension $<r$.
\end{proof}

\subsection{Proof of Theorem \ref{key}}
The rest of this section is devoted to the proof of Theorem \ref{key}. Our argument differs from that in \cite[\S 5.3]{arizona} in that we make no use of \cite[Proposition 5.2.1]{arizona}. Instead we use the results from \S 2 of this paper. 

\medskip

Fix $w$ and $x$ as in the statement of Theorem \ref{key}. As in the proof of Theorem \ref{main complex}, we write $\cF^i$ for $R^i\pi_{H,\ast}j_!(\cO_\cX^+/p)^a$. Recall the notion of a rank one point on the topological space of a locally spatial diamond from \cite[Definition 3.2.1]{extensions}; this means that the point has no proper generalizations. We further recall further a slight generalization of some comments from the end of \cite[\S 2.1]{arizona}. Let $Y$ be a spatial diamond, $S$ is a spectral space and $q : |Y| \to S$ a spectral map. If $T\sub S$ is quasicompact and closed under generalizations, then the preimage $q^{-1}(T)$ carries a natural structure of a spatial diamond, defined as follows: The conditions on $T$ are equivalent to $T = \bigcap_{T\sub U} U$, where $U$ runs through the quasicompact open subsets of $S$ containing $T$, by \cite[Tag 0A31]{stacks-project}. We may then define
\[
 q^{-1}(T) \defeq \varprojlim_{T \sub U} q^{-1}(U). 
\]
As mentioned, this naturally a spatial diamond by \cite[Proposition 11.18, Lemma 11.22]{diamonds}, with an injective quasi-pro-étale morphism $q^{-1}(T) \to Y$. When $T={\rm Gen}(s)$ is the set of generalizations of a point $s\in S$, we simply write $q^{-1}(s)$ for $q^{-1}({\rm Gen}(s))$. By \cite[Proposition 2.2.5, Lemma 2.2.6]{arizona}, we have
\[
 \cF_x^i \cong H_\et^i(\pi_H^{-1}(x), j_!(\cO_\cX^+/p)^a).
\]

\begin{lem}\label{rank one}
It suffices to prove Theorem \ref{key} for points $x$ with no proper generalizations.
\end{lem}

\begin{proof}
This is the precise analogue of \cite[Lemma 5.3.1]{arizona} in our situation, and the same proof works.
\end{proof}

From now on, we assume that $x\in |\Fl_{G,\mu}^w|/H$ has no proper generalizations. Let $\wt{x}\in |\Fl_{G,\mu}^w|$ be any lift of $x$; this is a rank one point. By Lemma \ref{schubert cell inside affine}, we have $\Fl_{G,\mu}^w \sub V_w$ and we let $x_w \defeq \wt{x}H_w \in |\Fl_{G,\mu}^w|/H_w$ (recall that $H_w \sub H$). Now consider the set $(\wt{x}H)/H_w \sub |\Fl_{G,\mu}^w|/H_w$ (which contains $x_w$), equipped with the subspace topology. On the one hand, $(\wt{x}H)/H_w$ consists only of points with no proper generalizations in the spectral space $|\Fl_{G,\mu}^w|/H_w$, so it is Hausdorff. It is also compact, since it is the image of $\wt{x}H \sub |\Fl_{G,\mu}|$. Finally, it is spectral, since it is the preimage of $x$ under the map $|\Fl_{G,\mu}^w|/H_w \to |\Fl_{G,\mu}^w|/H$. Therefore it is profinite, e.g.\ by \cite[Tag 0905, Lemma 5.23.7]{stacks-project}. By our earlier remarks, $\pi_{H_w}^{-1}((\wt{x}H)/H_w)$ is naturally a spatial diamond. Our first goal is to prove that $\pi_{H_w}^{-1}((\wt{x}H)/H_w)$ is an affinoid perfectoid space. As emphasized by a referee, the fact that $\Fl_{G,\mu}^w \sub V_w$ implies that
\[
\pi_{H_w}^{-1}((\wt{x}H)/H_w) \sub \ol{\cX}_{H_w,w} = \pi_{H_w}^{-1}(|V_w|/H_w).
\]
The latter is a perfectoid space by Corollary \ref{General noncanonical loci}, and that Corollary moreover shows that $\pi_{H_w}^{-1}((\wt{x}H)/H_w)$ is contained in some affinoid perfectoid open subset of $\ol{\cX}_{H_w,w}$. However, proving that $\pi_{H_w}^{-1}((\wt{x}H)/H_w)$ itself is affinoid perfectoid requires a more intricate argument. We prove this in Corollary~\ref{aff perf} below, but first we require some preliminary results. 

\begin{lem}\label{connected components}
Let $Y$ be a spatial diamond, let $S$ be a profinite set and assume that we have a spectral map $q : |Y| \to S$. If $q^{-1}(s)$ is affinoid perfectoid for all $s\in S$, then $Y$ is affinoid perfectoid.
\end{lem}

\begin{proof}
This is a direct consequence of \cite[Lemma 11.27]{diamonds}.
\end{proof}

As $x_w \in (\wt{x}H)/H_w$ is arbitrary by construction, it therefore suffices to prove that $\pi_{H_w}^{-1}(x_w)$ is affinoid perfectoid. To show this, we will use the following simple group-theoretical lemma.

\begin{lem}\label{central series}
In this lemma and its proof only, let $G$ be a group, let $H \sub G$ be a subgroup and let $Z_0 \sub Z_1 \sub \dots \sub Z_r = G$ be a sequence of normal subgroups of $G$ such that $Z_{i+1}/Z_i$ is central in $G/Z_i$. Then $HZ_0 \sub HZ_1 \sub \dots \sub HZ_r =G$ is a sequence of subgroups of $G$, and $HZ_i$ is normal in $HZ_{i+1}$.
\end{lem}

\begin{proof}
That the $HZ_i$ are subgroups is clear. For the final assertion, take $z_1 \in Z_i$, $z_2 \in Z_{i+1}$ and $h_1,h_2 \in H$. Since $Z_{i+1}/Z_i$ is central in $G/Z_i$, we can write $z_2 h_1 = h_1 z_2 w$ for some $w\in Z_i$. Then we compute
$$ (h_2 z_2)(h_1 z_1)(h_2 z_2)^{-1} = h_2 h_1 ( z_2 (wz_1) z_2^{-1}) h_2^{-1}, $$
which is in $HZ_i$ as desired.
\end{proof}

\begin{lem}\label{central series 2}
There is a sequence of normal subgroups $1=Z_0 \sub Z_1 \sub \dots \sub Z_r=\ol{U}$ such that $Z_{i+1}/Z_i$ is central in $\ol{U}/Z_i$ and each $Z_i$ is a product (as a subscheme, not necessarily as a subgroup) of root subgroups of $G$ with respect to $T$. Moreover, the product $\ol{U}_wZ_i$ is a product of root subgroups and $\ol{U}_wZ_i$ is normal in $\ol{U}_w Z_{i+1}$, and $\ol{U}_w Z_{i+1}/\ol{U}_w Z_{i}$ is isomorphic to $\bG_a^{d_i} $ for $d_i = \dim \ol{U}_wZ_{i+1}/\ol{U}_wZ_{i}$.
\end{lem}

\begin{proof}
The first part follows from \cite[Proposition 5.1.16(2)]{conrad}; one can take the $Z_i$ to be the groups called $\text{U}_{\geq n}$ in \emph{loc.cit.}, with $\Psi =\Phi^+$ in the notation of \emph{loc.cit}. The fact that $Z_{i+1}/Z_i$ is central in $\ol{U}/Z_i$ follows from the same calculation that shows that $\text{U}_{\geq n}/\text{U}_{\geq n+1}$ is abelian in the proof of \cite[Proposition 5.1.16]{conrad}.

\medskip

For the second part, first note that $\ol{U}_w Z_i$ is a product of root subgroups as a subscheme since $\ol{U}_w$ and $Z_i$ are, and it is a subgroup since $Z_i$ is normal. The normality of $\ol{U}_w Z_i$ in $\ol{U}_w Z_{i+1}$ then follows by applying Lemma \ref{central series} to the functors of points. Finally, $\ol{U}_w Z_{i+1}/\ol{U}_w Z_{i}$ is isomorphic to $Z_{i+1}/(Z_{i+1}\cap \ol{U}_w Z_i) \sub Z_{i+1}/Z_i$ and the latter is isomorphic, as groups, to the product of the root subgroups in $Z_{i+1}$ but not in $Z_i$ by construction (i.e. \cite[Proposition 5.1.16(2)]{conrad}), so $\ol{U}_w Z_{i+1}/\ol{U}_w Z_i$ is isomorphic as groups to the product of root subgroups in $Z_{i+1}$ but not in $\ol{U}_w Z_i$. 
\end{proof}

\begin{prop}\label{shrink to stabilizer}
Let $\Gamma_{x,w} \defeq Stab_{H_w}(\wt{x})$ and consider the point $\wt{x}_\Gamma = \wt{x}\Gamma_{x,w} \in |\Fl_{G,\mu}|/\Gamma_{x,w}$. Then the natural map
$$ \pi_{\Gamma_{x,w}}^{-1}(\wt{x}_\Gamma) \to \pi_{H_w}^{-1}(x_w)$$
is an isomorphism.
\end{prop}

\begin{proof}
We will apply Lemma \ref{stabilizers} repeatedly (once is not enough, since $\Gamma_{x,w}$ might not be normal in $H_w$). The group $H_w$ is nilpotent, so let $Z_0=1 \sub Z_1=Z(H_w) \sub Z_2 \sub \dots \sub Z_r = H_w$ be the upper central series of $H_w$ and set $H_i = \Gamma_{x,w}Z_i$. Consider the points $\wt{x}H_i \in |\Fl_{G,\mu}|/H_i$. We claim that the natural maps
\[
 \pi_{H_i}^{-1}(\wt{x}H_i) \to \pi_{H_{i+1}}^{-1}(\wt{x}H_{i+1}) 
 \]
are isomorphisms for all $i=0,\dots,r-1$. Since the composition of all these maps is the map in the statement of the Proposition, this suffices. But this is a direct application of Lemma \ref{stabilizers}, setting (in the notation of that lemma) $G=H_{i+1}/H_i$, $X=\ocX_{H_i}$, $S = |\pi_{H_i}^{-1}(\wt{x} H_i)|$, $X \times_{\us{X}} \ul{S} = \pi_{H_i}^{-1}(\wt{x} H_i)$ (note that $H_i$ is normal in $H_{i+1}$ by Lemma \ref{central series}, and that $X/G$ can be identified with $\ocX_{H_{i+1}}$ by Lemma \ref{shimura variety quotient}). 
\end{proof}

We now study $\pi_{\Gamma_{x,w}}^{-1}(\wt{x}_\Gamma)$. First, consider $\pi_{\mbf{1}}^{-1}(\wt{x})$. Recall the standard affinoid open $\cV_\mbf{1} \sub \Fl_{G,\mu}$ from \S \ref{perfectoid loci}. By Corollary \ref{General noncanonical loci} and the definitions, we have
\[
 \pi_{\mbf{1}}^{-1}(\wt{x}) \sub \pi_{\mbf{1}}^{-1}(V_w) = \bigcup_{k^\prime \geq 0} \pi_{\mbf{1}}^{-1}(\cV_\mbf{1} \gamma^{-k^\prime}w^{-1}) = \bigcup_{k\geq 0} \ocX_{\mbf{1}}(\epsilon)_a\gamma^{-k}w^{-1}.
 \]
It follows that we may choose $k$ and $k^\prime$ such that $\wt{x} \sub \cV_\mbf{1} \gamma^{-k^\prime}w^{-1}$ and
\[
 \pi_{\mbf{1}}^{-1}(\wt{x}) \sub \ocX_{\mbf{1}}(\epsilon)_a\gamma^{-k}w^{-1} \sub \pi_{\mbf{1}}^{-1}(\cV_\mbf{1} \gamma^{-k^\prime}w^{-1}).
\]
Using that $\cV_\mbf{1}$ is stable under $\ol{P}_\mu(\Zp)$ and that $\gamma^{k^\prime}\ol{P}_\mu(\Zp)\gamma^{-k^\prime} \supseteq \ol{P}_\mu(\Zp)$, we see that $\cV_\mbf{1}\gamma^{-k^\prime}w^{-1}$ is stable under $\ol{P}_{\mu,w}(\Zp)$, and hence a fortiori stable under $H_w$ and $\Gamma_{x,w}$. Since
$\Gamma_{x,w}$ stabilizes the rank one point $\wt{x}$, we may choose a basis of open neighbourhoods $U_t$, $t\in T$ (some index set) of $\wt{x}$ such that, for all $t$, $U_t$ is a $\Gamma_{x,w}$-stable rational subset of $\cV_\mbf{1}\gamma^{-k^\prime}w^{-1}$
and $\pi_{\mbf{1}}^{-1}(U_t) \sub \ocX_{\mbf{1}}(\epsilon)_a\gamma^{-k}w^{-1}$.
(Any rational subset of $\cV_\mbf{1}\gamma^{-k^\prime}w^{-1} $has finitely many $\Gamma_{x,w}$-translates by \cite[Lemma 2.2]{scholze-lt} or \cite[Lemma 5.2.2]{arizona}, so their
intersection is again a rational subset.)
Then we see that
\[
 \wt{x}_\Gamma = \bigcap_{t\in T}|U_t|/\Gamma_{x,w} 
\]
and hence that
\[
 \pi_{\Gamma_{x,w}}^{-1}(\wt{x}_\Gamma) = \varprojlim_{t\in T} \pi_{\Gamma_{x,w}}^{-1}(|U_t|/\Gamma_{x,w}),
 \]
so to prove that $\pi_{\Gamma_{x,w}}^{-1}(\wt{x}_\Gamma)$ is affinoid perfectoid it suffices to prove that each $\pi_{\Gamma_{x,w}}^{-1}(|U_t|/\Gamma_{x,w})$ is affinoid perfectoid. For convenience, we introduce the principal congruence subgroups
\[
 K(p^m) \defeq \{g\in G(\Zp) \mid (g\text{ mod }p) =1  \}.
 \]
Note that we may write
\[
 \pi_\mbf{1}^{-1}(U_t) = \varprojlim_{m \geq 0} \pi_{\Gamma_{x,w} \cap K(p^m)}^{-1}(|U_t|/(\Gamma_{x,w} \cap K(p^m))),
 \]
and that $\pi_\mbf{1}^{-1}(U_t)$ is a rational subset of $\ocX_\mbf{1}(\epsilon)_a \gamma^{-k}w^{-1}$ by construction. Since rational subsets come from finite level in an inverse limit, it follows that there is an $m$ such that $\pi_{\Gamma_{x,w} \cap K(p^m)}^{-1}(|U_t|/(\Gamma_{x,w} \cap K(p^m)))$ is a rational subset of $|\ocX_{\mbf{1}}(\epsilon)_a| \gamma^{-k}w^{-1}/(\Gamma_{x,w} \cap K(p^m))$ (which is affinoid perfectoid by Corollary \ref{General noncanonical loci}), hence affinoid perfectoid. Let $G_m \defeq \Gamma_{x,w} / (\Gamma_{x,w} \cap K(p^m))$; this is a finite group. Note that $\pi_{\Gamma_{x,w}}^{-1}(|U_t|/\Gamma_{x,w})$, being a quasicompact open subset of $|\cX^*_{\mbf{1}}(\epsilon)_a| \gamma^{-k}w^{-1}/\Gamma_{x,w}$, is a quasicompact perfectoid space.

\begin{prop}\label{stab aff perf}
$\pi_{\Gamma_{x,w}}^{-1}(|U_t|/\Gamma_{x,w})$ is the quotient of $\pi_{\Gamma_{x,w} \cap K(p^m)}^{-1}(|U_t|/(\Gamma_{x,w} \cap K(p^m)))$ by the finite group $G_m$.  Hence it is affinoid perfectoid.
\end{prop}

\begin{proof}
The first claim follows from Proposition \ref{shimura variety quotient}, and then the second
follows from Proposition \ref{surjection finite level} and \cite[Thm.~3.3.25]{kedlaya-liu-2}.
\end{proof}

We now summarize our discussion so far.

\begin{cor}\label{aff perf}
$\pi_{H_w}^{-1}((\wt{x}H)/H_w)$ is affinoid perfectoid with Zariski closed boundary.
\end{cor}

\begin{proof}
We begin by showing that $\pi_{H_w}^{-1}((\wt{x}H)/H_w)$ is affinoid perfectoid, summarizing the arguments above. First, by Lemma \ref{connected components} it is enough to show that $\pi_{H_w}^{-1}(x_w)$ is affinoid perfectoid, which by Proposition \ref{shrink to stabilizer} is equivalent to showing that $\pi_{\Gamma_{x,w}}^{-1}(|U_t|/\Gamma_{x,w})$ is affinoid perfectoid, and this is Proposition \ref{stab aff perf}. That the boundary is Zariski closed is then immediate.
\end{proof}

\begin{rem}\label{differences}
Let us now make a few remarks on the differences between the arguments of this paper and that of \cite{arizona}. Thanks to the results in \S \ref{sec: prelim}, our results of the fibers of $\pi^{-1}_{H_w}$ are slightly stronger in the sense that we can prove Corollary \ref{aff perf}; in \cite{arizona} we could essentially only prove that it is affinoid perfectoid after a possible modification of the boundary, which was enough to deduce the cohomological consequences that we needed. We also note that Proposition \ref{shrink to stabilizer}, where Lemma \ref{stabilizers} is the key ingredient, allows us to bypass the technical but powerful \cite[Proposition 5.2.1]{arizona}. Indeed for our argument here it suffices to produce neighbourhoods of points which are stable under the stabilizer of the point rather than stable under the whole group. The former is almost trivial in comparison to the latter.
\end{rem}

From here, we wish to compute cohomology on $\pi_H^{-1}(x)$ using the cover $\pi_{H_w}^{-1}((\wt{x}H)/H_w)$. Again, we are faced with the issue that $H_w$ may not be normal in $H$, but the fact that $H$ is nilpotent comes to the rescue.
Recall that, by Lemma \ref{rank one} and the paragraph preceding it, we need to show that
\[
 H^i_{\et}(\pi_H^{-1}(x), j_! (\cO_{\cX}^+/p)^a) = 0
\]
for all $i > d- \dim \Fl^w_{G,\mu}$.
Choose a series of subgroups $1=Z_0 \sub \dots \sub Z_r=\ol{U}$ as in Lemma \ref{central series 2}, and set $H_{w,i} = H \cap \ol{U}_w(\Zp)Z_i(\Zp)$ for $i=0,\dots,r$. Note that we have $H_{w,0}=H_w$ and $H_{w,r}=H$. Moreover, $H_{w,i}$ is normal in $H_{w,i+1}$ and $H_{w,i+1}/H_{w,i}$ embeds naturally as a finite index subgroup of
$(\ol{U}_wZ_{i+1}/\ol{U}_w Z_{i})(\Zp)=\ol{U}_wZ_{i+1}(\Zp)/\ol{U}_w Z_{i}(\Zp)$ (this equality follows by looking at the decomposition into root spaces), so it is isomorphic to $\Zp^{d_i}$ (in the notation of Lemma \ref{central series 2}). By \cite[Theorem 2.2.7]{arizona}, we have Hochschild--Serre spectral sequences with $E_2$-page
\[
 H^s_{cts}(H_{w,i+1}/H_{w,i}, H_\et^t(\pi_{H_{w,i}}^{-1}(\wt{x}H/H_{w,i}), j_!(\cO_\cX^+/p)^a))$$ 
$$\implies H_\et^{s+t}(\pi_{H_{w,i+1}}^{-1}(\wt{x}H/H_{w,i+1}), j_!(\cO_\cX^+/p)^a)
\]
for all $i=0,\dots,r-1$. By Corollary \ref{aff perf} and \cite[Proposition 5.1.4]{arizona}, we have 
\[
 H_\et^t(\pi_{H_{w}}^{-1}(\wt{x}H/H_{w}), j_!(\cO_\cX^+/p)^a) =0 
 \]
for all $t > 0$. 

\begin{rem}
Here we use the fact that the notions of Zariski closed and strongly Zariski closed subsets of affinoid perfectoid spaces are now known to agree, by \cite[Remark 7.5]{bhatt-scholze}. It is perhaps worth noting that, just like in \cite{arizona}, one may prove directly that the sets relevant to our arguments are strongly Zariski closed; the general result from \cite{bhatt-scholze} is strictly speaking not needed.
\end{rem}

From this, we deduce by induction using the Hochschild--Serre spectral sequences above and the fact that $H_{w,i+1}/H_{w,i} \cong \Zp^{d_i}$ has cohomological dimension $d_i$ for continuous cohomology that 
$$ H^i_{\et}(\pi_H^{-1}(x), (j_! \cO_{\cX}^+/p)^a) = 0 $$
for all $i > d_0 + \dots + d_{r-1} = \dim \ol{U} - \dim \ol{U}_w$. We then finish the proof of Theorem \ref{key} by noting that, by Lemma \ref{dimension count},
$$ \dim \ol{U} - \dim \ol{U}_w = \dim \ol{N}_\mu - \dim \Fl_{G,\mu}^w = d -  \dim \Fl_{G,\mu}^w. $$

\section{Bounds on codimensions for ordinary parts}\label{sec: ordinary parts}

In this section, we deduce an application of Theorem~\ref{main complex}, namely we bound
the codimension over the Iwasawa algebra 
of the ordinary part of completed homology of the Shimura varieties 
for $G$. The application is Theorem~\ref{homology vanishing} and it relies on a Poincar\'e duality spectral sequence for 
the ordinary part of the homology of locally symmetric spaces; 
most of the work in this section is in constructing this spectral sequence. 

We work with a group $G/\Q$ which is split at $p$ since we can only prove our main result, Theorem \ref{homology vanishing} under this assumption (and the additional assumption that $G$ admits a Shimura variety of Hodge type), though we believe that the weaker assumption that $G$ is only quasi-split at $p$ should suffice for sections \ref{sec: completed} to \ref{subsec:ordinary parts}. We
also do not need to know that $G$ admits a Shimura variety of Hodge type until we appeal to Theorem~\ref{main complex}, as
we only need to consider the locally symmetric spaces associated to $G$. 

We note that the existence of such a spectral sequence has been previously announced by Emerton, 
relying on his theory of ordinary parts~\cite{emerton-ord1, emerton-ord2}. In this paper, we give a different construction relying on computations of (co)homology using singular and simplicial chains, in the style of Ash--Stevens \cite{ash-stevens}, as well as ideas of Hill \cite{hill}.

\subsection{Completed cohomology and distributions} \label{sec: completed} 
We begin by recalling some standard material; the reader is referred to \cite{borel-serre} for more details and to \cite[\S 3.1]{newton-thorne} for a useful summary. Let $G$ be a connected reductive group over $\Q$ and assume it admits a flat affine model
over $\Z$ which is split at $p$. Let $X$ denote the symmetric space for 
$G(\R)$ in the sense of~\cite[\S 2]{borel-serre}. Let $D\defeq \dim_{\R} X$. 

For simplicity, we will consider neat compact open subgroups $K\subset G(\A_f)$ which 
are of the form $K = \prod_{\ell} K_\ell$, 
where $\ell$ runs over finite primes and $K_\ell\subseteq G(\Z_{\ell})$. Let $G^{\mathrm{ad}}$ denote the adjoint group of $G$, and let $G^{\mathrm{ad}}(\R)^+$ denote the identity component of $G^{\mathrm{ad}}(\R)$ in the real topology. We let $G(\R)_+$ denote the preimage of $G^{\mathrm{ad}}(\R)^+$ under the natural map $G(\R) \to G^{\mathrm{ad}}(\R)$, and we set $G(\Q)_+ = G(\Q) \cap G(\R)_+$. For $K$ as above, we define a locally symmetric space
\[
X_K\defeq G(\Q)_+\backslash X\times G(\A_f) / K.   
\]
As is well known, the fact that $K$ is neat implies that $X_K$ is a smooth manifold; $X_K$ is orientable because $X$ is and $G(\R)_+$ acts by orientation-preserving maps.

\begin{rem}
The reader familiar with the literature on adelic locally symmetric spaces will know that other sources use slightly different definitions. For example, it is common to quotient out on the left either by the full $G(\Q)$ or the smaller $G(\Q) \cap G(\R)^+$, where $G(\R)^+ \sub G(\R)$ is the identity component. Roughly speaking, the difference between these definitions is only in the structure of the set of components and the theory developed here could be developed for either of these choices (at least for sufficiently small $K$). The reason for our convention is that, when $G$ admits a Shimura datum of Hodge type, $X_K$ is exactly the complex points of the Shimura variety for $G$ with level $K$. This follows from \cite[Lemma 5.11]{milne-notes} and the fact that in this case, the maximal $\R$-split torus in the center of $G$ is $\Q$-split.
\end{rem}
Following \cite{newton-thorne}, we will also make use of the space
\[
\frakX \defeq G(\Q)_+ \backslash X \times G(\A_f),
\]
where $G(\A_f)$ is given the discrete topology prior to taking the quotient. As a topological space, this is an uncountable disjoint union of copies of $X$.

The symmetric space $X$ admits a partial (Borel--Serre) compactification
$X^{\BS}$, to which the $G(\Q)$-action extends, and we 
set 
\[
X^{\BS}_K\defeq G(\Q)_+ \backslash X^{\BS}\times G(\A_f) / K;
\]
this is a compactification of $X_K$ and the inclusion $X_K\hookrightarrow X_K^{\BS}$ is a homotopy equivalence. 
We also consider the boundaries $\partial X^{\BS}\defeq X^{\BS} \setminus X$ and 
$\partial X^{\BS}_K\defeq X^{\BS}_K \setminus X_K$. We also consider the space
\[
\frakX^{\BS} \defeq G(\Q)_+ \backslash X^{\BS} \times G(\A_f)
\]
where again $G(\A_f)$ is given the discrete topology. We set $\partial \frakX^{\BS} = \frakX^{\BS} \setminus \frakX$, and note that all these spaces carry right actions of $G(\A_f)$, whose restrictions to neat compact opens $K\sub G(\A_f)$ are free. In other words, $\frakX \to X_K$ etc. are $K$-covers.
We let $C_{\bullet}$ and $\partial C_{\bullet}$ denote the complexes of singular 
chains with $\Z$-coeffcients on $\frakX^\BS$ and $\partial \frakX^{\BS}$ respectively; these are right $\Z[G(\A_f)]$-modules. 
There is a natural morphism of complexes $\partial C_{\bullet}\to C_{\bullet}$ 
and we define $C^{\BM}_{\bullet}$ to be the cone of this morphism. 
Given a left $\Z[K]$-module $M$, we define 
\[
C_{\A,\bullet}(K,M)\defeq C_{\bullet}\otimes_{\Z[K]} M
\]
and 
\[
C^{\BM}_{\A,\bullet}(K,M)\defeq C^{\BM}_{\bullet}\otimes_{\Z[K]} M.
\] 
These complexes compute the homology, respectively
the Borel--Moore homology, of $X_{K}$ with coefficients in the local system $\widetilde{M}$ determined by $M$ via the $K$-cover $\frakX^{\BS} \to X^{\BS}_K$; we record this as a proposition.

\begin{prop}\label{prop:fin level homology} Let $M$ be a left $K$-module. 
There are canonical isomorphisms 
\[
H_*(X_K, \widetilde{M})\simeq H_*\left(C_{\A,\bullet}(K,M)\right)\ \mathrm{and}
\ H^{\BM}_*(X_K, \widetilde{M})\simeq H_*\left(C^{\BM}_{\A,\bullet}(K,M)\right)
\]
\end{prop}

\begin{remark}\label{rem:fin complexes} 
Choose a finite (combinatorial) triangulation of $X_K^{\BS}$ such that $\partial X^{\BS}_K$ is a simplicial subcomplex ($X_K^{\BS}$ is homeomorphic to a smooth compact manifold with boundary \cite[\S 11.1]{borel-serre}, so this can be done using e.g. \cite[Theorem 10.6]{munkres}). Pulling back these triangulations to $\frakX^{\BS}$ and $\partial \frakX^{\BS}$ we obtain bounded complexes $F_\bullet$ and $\partial F_\bullet$ of finite free right $\Z[K]$-modules which are homotopy equivalent to $C_\bullet$ and $\partial C_\bullet$, respectively. We let $F_\bullet^{\BM}$ denote the cone of $\partial F_\bullet \to F_\bullet$, which is then homotopy equivalent to $C_\bullet^{\BM}$.

This allows us to construct complexes 
$C_{\bullet}(K, M)\ \mathrm{and}\ C^{\BM}_{\bullet}(K,M)$
which have good finiteness properties and are homotopy equivalent to 
$C_{\A, \bullet}(K, M)\ \mathrm{and}\ C^{\BM}_{\A, \bullet}(K,M)$, respectively\footnote{The resulting complexes are called Borel--Serre complexes in~\cite{hansen-thesis}, though they are constructed slightly differently there.}, by setting
\[
C_{\bullet}^{(\BM)}(K,M)\defeq F_{\bullet}^{(\BM)}\otimes_{\Z[K]}M.
\]
Here and elsewhere we write $(\BM)$ to mean that one can either make the same construction for the complexes with no superscript or with the $\BM$ superscript.
\end{remark}

Analogous constructions can be made for cohomology and cohomology
with compact support. If $M$ is a right $K$-module, define 
\[
C^{\bullet,(\BM)}_{\A}(K,M)\defeq \Hom_{\Z[K]}\left(C^{(\BM)}_{\bullet}, M\right).
\]
We have the following analogue of Proposition~\ref{prop:fin level homology}, which once again is an instance of descent. 

\begin{prop}Let $M$ be a right $K$-module. There are canonical isomorphisms 
\[
H^*(X_K, \widetilde{M})\simeq H^*\left(C^\bullet_{\A}(K,M)\right)\ \mathrm{and}
\ H^*_c(X_K, \widetilde{M})\simeq H^*\left(C^{\bullet,\BM}_{\A}(K,M)\right)
\]
\end{prop}

We also have complexes $C^{\bullet,(\BM)}(K,M)= \Hom_{\Z[K]}\left(F^{(\BM)}_{\bullet}, M\right)$ with good finiteness properties as in Remark~\ref{rem:fin complexes}. 

We now discuss the action of $G(\A_f)$ on the adelic complexes we have defined above. 
Let $S$ be a finite set of finite primes, not necessarily containing $p$. 
Assume that $M$ is a left $\Z[G(\A^S)\times K_S]$-module. For any $g\in G(\A^S)$, 
we have an isomorphism 
\[
g_*: C^{(\BM)}_{\A, \bullet}(K,M)\toisom C^{(\BM)}_{\A,\bullet}(g^{-1}Kg,M), \,\,\,\, 
\sigma \otimes m \mapsto \sigma g \otimes g^{-1}m.
\]
This can be translated into a Hecke action of the double coset $[KgK]$ for
$g\in G(\A^S)$ by taking the composition 
\begin{equation}\label{eq:hecke fin homology}
C^{(\BM)}_{\A, \bullet}(K, M)\to C^{(\BM)}_{\A, \bullet}(K\cap gKg^{-1}, M)\stackrel{g_*}{\toisom} 
C^{(\BM)}_{\A,\bullet}(g^{-1}Kg\cap K, M)\to C^{(\BM)}_{\A,\bullet}(K, M),
\end{equation}
where the first morphism is the trace map
\[
\sigma \otimes m \mapsto \sum \sigma k \otimes k^{-1}m,
\] 
where the sum runs over a set of coset representatives $k$ for $K/(K\cap gKg^{-1})$, 
and the last morphism is restriction $\sigma \otimes m \mapsto \sigma \otimes m$. 
One can check that this recovers the usual Hecke action on homology /
Borel--Moore homology. 

Assume now that $M$ is a right $\Z[G(\A^S)\times K_S]$-module. For
any $g\in G(\A^S)$, we have an isomorphism 
\[
g^*: C^{\bullet, (\BM)}_{\A}(K,M)\toisom C^{\bullet, (\BM)}_{\A}(gKg^{-1},M), \,\,\,\,
g^*(\phi)(\sigma) = \phi(\sigma g)g^{-1}. 
\]
This can be translated into a Hecke action of the double coset $[KgK]$ for
$g\in G(\A^S)$ by taking the composition 
\begin{equation}\label{eq:hecke fin cohomology}
C^{\bullet, (\BM)}_{\A}(K, M)\to C^{\bullet, (\BM)}_{\A}(K\cap g^{-1}Kg, M)\stackrel{g^*}{\toisom} 
C^{\bullet, (\BM)}_{\A}(gKg^{-1}\cap K, M)\to C^{\bullet, (\BM)}_{\A}(K, M),
\end{equation}
where the first morphism is restriction and the third morphism is the trace map 
\[
\phi \mapsto \psi(\sigma):= \sum \phi(\sigma k)k^{-1},
\]
where the sum runs over a set of coset representatives $k$ for $K/(K\cap g^{-1}Kg)$. 
One can check that this recovers the usual Hecke action on cohomology / cohomology
with compact support. 

If $R$ is a commutative ring and $M,N$ are $R$-modules, we have a canonical isomorphism 
\begin{equation}\label{eq:universal coefficient isom}
\mathrm{RHom}_{R}\left(C^{(\BM)}_{\A,\bullet}(K,M), N\right) \simeq 
C^{\bullet, (\BM)}_{\A}\left(K, \mathrm{RHom}_R(M,N)\right),
\end{equation} 
from the adjunction between tensor products and homomorphisms. For example,
if $R= N = \Z$, we obtain the universal coefficient isomorphism between 
homology and cohomology, respectively between Borel--Moore homology and
cohomology with compact support. One can check from the explicit 
descriptions~\eqref{eq:hecke fin homology} and~\eqref{eq:hecke fin cohomology} that the universal coefficient
isomorphism is equivariant for the Hecke action of $[KgK]$.

Our goal is now to use these explicit adelic complexes to describe ordinary 
completed (co)homology. Let $K'_p\subseteq K_p$ be a compact open subgroup; set $K':=K^pK'_p$, 
this is a compact open subgroup of $K$. 

\begin{lemma}\label{lem:fin level distribution}\leavevmode
\begin{enumerate}
\item For any $n\in \Z_{\geq 1}$, there is a canonical isomorphism 
\[
C^{(\BM)}_{\A,\bullet}(K', \Z/p^n\Z)\simeq C^{(\BM)}_{\A,\bullet}\(K, \Z/p^n\Z[K_p/K'_p]\), 
\]
where $K^p$ acts trivially and $K_p$ acts by left translation on $\Z/p^n\Z[K_p/K'_p]$. 

\item As $K^p$ varies, this isomorphism is equivariant for the action of $g^p\in G(\A^p_f)$. 

\item If $g\in G(\Q_p)$ and $K_{i,p} \subseteq K_p$ for $i=1,2$ with $g^{-1}K_{1,p}g\subseteq K_{2,p}$ 
 and $K_{i}\defeq K^pK_{i,p}$, the morphism 
\[
g_*: C^{(\BM)}_{\A,\bullet}(K_{1}, \Z/p^n)\to C^{(\BM)}_{\A,\bullet}(K_{2}, \Z/p^n), \,\,\, g_*(\sigma\otimes \lambda) = \sigma g\otimes \lambda
\] 
corresponds to the morphism 
\[
g_*:C^{(\BM)}_{\A,\bullet}\(K, \Z/p^n[K_p/K_{1,p}]\)\to C^{(\BM)}_{\A,\bullet}\(K, \Z/p^n[K_p/K_{2,p}]\), 
\]
\[
g_*(\sigma \otimes \lambda k)= \sigma kg \otimes \lambda.
\]
\item If $K_{2,p}\subseteq K_{1,p}\subseteq K_{p}$, and $K_{i}\defeq K^pK_{i,p}$ for $i=1,2$, the trace morphism 
$C^{(\BM)}_{\A,\bullet}(K_{1}, \Z/p^n)\to C^{(\BM)}_{\A,\bullet}(K_{2},\Z/p^n)$ corresponds to the
morphism 
\[
C^{(\BM)}_{\A,\bullet}(K, \Z/p^n[K_p/K_{1,p}])\to C^{(\BM)}_{\A,\bullet}(K, \Z/p^n[K_p/K_{2,p}])
\]
induced by the natural trace map $\mathrm{tr}: \Z/p^n[K_p/K_{1,p}]\to \Z/p^n[K_p/K_{2,p}]$. 
\end{enumerate}
\end{lemma}

\begin{proof} The first assertion is a simple adjunction, but to make the verification of formulas easier for the reader we give explicit isomorphisms. So, define morphisms 
\[
\iota: C^{(\BM)}_{\A,\bullet}(K', \Z/p^n)\to C^{(\BM)}_{\A,\bullet}\(K, \Z/p^n[K_p/K'_p]\)
\] 
given by $\sigma \otimes \lambda \mapsto \sigma \otimes \lambda$ 
and 
\[
\eta: C^{(\BM)}_{\A,\bullet}\(K, \Z/p^n[K_p/K'_p]\)\to C^{(\BM)}_{\A,\bullet}(K', \Z/p^n)
\] 
given by $\sigma \otimes \lambda k \mapsto \sigma k \otimes \lambda$. 
It is not hard to check that $\iota$ and $\eta$ are well-defined, 
and that they are mutual inverses. 
The Hecke equivariance away from $p$ is immediate, and the Hecke-equivariance at $p$ 
can be checked by direct computation:
\[
(\iota_2\circ g_*\circ \eta_1)(\sigma \otimes \lambda k) = \sigma kg \otimes \lambda. 
\] 
Finally, the last assertion can also be checked by direct computation. 
\end{proof}

We now begin to describe ordinary completed homology in this language.
We fix the tame level $K^p$ and denote all locally symmetric spaces of this fixed tame 
level by $X_{K_p}$, where $K_p\subseteq G(\Z_p)$ is a compact open subgroup. 
Similarly, we will often omit the tame level when using the complexes
defined above; for example, we will write $C^{(\BM)}_{\AA,\bullet}(K_p,M)$
for $C^{(\BM)}_{\AA,\bullet}(K^p K_p, M)$.
Recall the assumption 
that $G_{\Z_p}$ is split.
Choose a split Borel subgroup $B\subset G_{\Z_p}$, 
with Levi decomposition $B = T\ltimes U$, where $T$ 
is the split torus and $U$ is the unipotent subgroup, with opposite unipotent subgroup $\overline{U}$. 
Let $T_0\defeq T(\Z_p)$, and, for $i\in \Z_{\geq 1}$, set $T_i\defeq \ker(T(\Z_p)\to T(\Z/p^i))$. Similarly, for $j\in \Z_{\geq 1}$, 
define $N_j\defeq \ker(U(\Z_p)\to U(\Z/p^j))$ and  $\overline{N}_j\defeq \ker(\overline{U}(\Z_p)\to \overline{U}(\Z/p^j))$. 
For $j\geq 0$ and $i\geq \max(j,1)$, set $K_{ij}\defeq \overline{N}_{i}T_{j}N_1\subset G(\Z_p)$. To simplify notation, we will also set $K_1 \defeq K_{10}$.

Define the completed homology / Borel--Moore homology at level $N_1$ by 
\[
\widetilde{H}^{(\BM)}_{*}(N_1)\defeq \varprojlim_{i,j,n} H^{(\BM)}_{*}(X_{K_{ij}}, \Z/p^n).
\]
This is equipped with a Hecke action of $[K^pg^pK^p]$ for $g^p\in G(\A_f^p)$ in the usual way. 
There is also a Hecke action at $p$. Let
\[
T^+\defeq \{t\in T(\Q_p)\mid t N_1 t^{-1} \subseteq N_1\}.  
\]
Also define 
\[
T^-\defeq \{t\in T(\Q_p)\mid t \overline{N}_1 t^{-1} \subseteq \overline{N}_1\}.  
\]
For $t\in T^+$, set $\overline{N}_t\defeq t^{-1}\overline{N}_1t$, and $N_t\defeq tN_1 t^{-1}$. 
For any $t\in T^+$, we have a Hecke action of $[N_1tN_1]$ on $\widetilde{H}^{(\BM)}_{*}(N_1)$ 
given by the composition 
\[
\widetilde{H}^{(\BM)}_{*}(N_1)\stackrel{\mathrm{tr}}{\to} \widetilde{H}^{(\BM)}_{*}(N_t)\stackrel{t_*}{\toisom} \widetilde{H}^{(\BM)}_{*}(N_1),
\]
where the first map is the natural trace map. Here $\tH^{(\BM)}_{*}(N_t)$ is defined as
\[
\widetilde{H}^{(\BM)}_{*}(N_t)\defeq \varprojlim_{i,j,n} H^{(\BM)}_{*}(X_{tK_{ij}t^{-1}}, \Z/p^n).
\]
The Hecke action above is compatible with the Hecke action 
of $[K_{ij}tK_{ij}]$ on each $H^{(\BM)}_{*}(X_{K_{ij}}, \Z/p^n\Z)$ as described in~\eqref{eq:hecke fin homology}. 

We wish to describe $\widetilde{H}^{(\BM)}_{*}(N_1)$ with its Hecke actions in terms of 
the adelic complexes defined above, with coefficients in a certain algebra of $\Z_p$-valued distributions. 
If $X$ is a profinite set, let $\mathscr{D}(X)$ denote the space of continuous $\Z_p$-valued distributions
on $X$; if $X$ is a profinite group then this carries a natural $\Zp$-algebra structure. In particular, we have 
\begin{equation}\label{eq:inverse limit}
\mathscr{D}(K_1/N_1) = \Z_p[\![K_1/N_1]\!] = \varprojlim_{i,j,n} \Z/p^n [K_1/K_{ij}]. 
\end{equation}

Set $K'_t\defeq K_1\cap tK_1t^{-1}$ and $K_t\defeq K_1\cap t^{-1}K_1t$. For $t\in T^+$, we have an isomorphism
\begin{equation}\label{eq:t-action}
\iota_{K_t,K_1}\circ t_*\circ \eta_{K_1,K'_t}: C^{(\BM)}_{\A,\bullet}(K_1,\mathscr{D}(K_1/N_t))\toisom C^{(\BM)}_{\A,\bullet}(K_1,\mathscr{D}(K_1/N_1)),
\end{equation}
that we describe explicitly as the composition of three isomorphisms. The first is an isomorphism 
\[
\eta_{K_1,K'_t}: C^{(\BM)}_{\A,\bullet}(K_1,\mathscr{D}(K_1/N_t))\toisom C^{(\BM)}_{\A,\bullet}(K'_t,\mathscr{D}(K'_t/N_t)).
\] 
Choose a set $S$ of coset representatives for $K_1/K'_t$. Any $\mu\in \mathscr{D}(K_1/N_t)$
can be written uniquely as $\sum_{s\in S}s\mu_s$ with $\mu_s\in \mathscr{D}(K'_t/N_t)$. Set 
\[
\eta_{K_1,K'_t}( \sigma \otimes \mu) \defeq \sum_{s\in S} \sigma s\otimes \mu_s.
\]
This is well-defined, independent of the choice of $S$ and it is an isomorphism since it has the inverse 
$\iota_{K'_t, K_1}(\sigma \otimes \mu) \defeq \sigma \otimes \mu$.
The second is the isomorphism 
\[
t_*: C^{(\BM)}_{\A,\bullet}(K'_t,\mathscr{D}(K'_t/N_t))\toisom C^{(\BM)}_{\A,\bullet}(K_t,\mathscr{D}(K_t/N_1))
\]
given by $\sigma \otimes \mu \mapsto \sigma t \otimes t^{-1}\mu t$. Finally,
the third is the isomorphism 
\[
\iota_{K_t, K_1}: C^{(\BM)}_{\A,\bullet}(K_t,\mathscr{D}(K_t/N_1))\toisom C^{(\BM)}_{\A,\bullet}(K_1,\mathscr{D}(K_1/N_1)).
\]
Note that $K_1/K'_t\simeq N_1/N_t$. We obtain 
\[
\iota_{K_t,K_1}\circ t_*\circ \eta_{K_1,K'_t}( \sigma \otimes \mu) = \sum_{s\in N_1/N_t} \sigma st \otimes t^{-1}\mu_s t. 
\]

\begin{thm}\label{thm:completed hom via distributions} There exist canonical isomorphisms
\[
\widetilde{H}^{(\BM)}_{*}(N_1)\simeq H_{*}\left(C^{(\BM)}_{\A,\bullet}(K_1,\mathscr{D}(K_1/N_1))\right).
\]
These are equivariant for the Hecke action of $[K^pg^pK^p]$ for all $g^p\in G(\A^p_f)$.
For $t\in T^+$, the Hecke action of $[N_1tN_1]$ on $\widetilde{H}^{(\BM)}_{*}(N_1)$ is induced 
from the composition 
\[
C^{(\BM)}_{\A,\bullet}(K_1,\mathscr{D}(K_1/N_1))\stackrel{\mathrm{tr}}{\to} C^{(\BM)}_{\A,\bullet}(K_1,\mathscr{D}(K_1/N_t))
\toisom C^{(\BM)}_{\A,\bullet}(K_1,\mathscr{D}(K_1/N_1)),
\]
where $\mathrm{tr}$ is induced from the natural trace map $\mathscr{D}(K_1/N_1)\to \mathscr{D}(K_1/N_t)$
and the isomorphism is given by $\iota_{K_t,K_1}\circ t_*\circ \eta_{K_1,K'_t}$. 
\end{thm}

\begin{proof} We claim first that the natural map of complexes 
\[
C^{(\BM)}_{\A, \bullet}(K_1, \mathscr{D}(K_1/N_1))\to \varprojlim_{i,j,n}C^{(\BM)}_{\A,\bullet}(K_1, \Z/p^n[K_1/K_{ij}])
\]
is a homotopy equivalence. Up to homotopy equivalences\footnote{These can be chosen to be functorial in the coefficients, and compatible with the transition morphisms between different levels.}, 
we can replace the above complexes with the corresponding Borel--Serre complexes as in Remark~\ref{rem:fin complexes}, and we have a natural map 
\[
C^{(\BM)}_{\bullet}(K_1, \mathscr{D}(K_1/N_1))\to \varprojlim_{i,j,n}C^{(\BM)}_{\bullet}(K_1, \Z/p^n[K_1/K_{ij}]),
\]
which can be seen by inspection to be an isomorphism using~\eqref{eq:inverse limit}.

We now claim that the homology of the complex $\varprojlim_{i,j,n}C^{(\BM)}_{\A,\bullet}(K_1, \Z/p^n[K_1/K_{ij}])$ 
computes $\widetilde{H}^{(\BM)}_{*}(N_1)$. By combining Lemma~\ref{lem:fin level distribution} and 
Proposition~\ref{prop:fin level homology}, we have an isomorphism 
\[
H^{(\BM)}_*(X_{K_{ij}}, \Z/p^n\Z)\toisom H_*\(C^{(\BM)}_{\A,\bullet}(K_1, \Z/p^n[K_1/K_{ij}])\)
\]
for every $n\in\Z_{\geq 1}$ and at each finite level $K_{ij}$. To conclude, we replace each finite level 
complex by the corresponding Borel--Serre complex. We obtain complexes of
abelian compact Hausdorff groups with continuous differentials. We conclude by noting
that the category of abelian compact Hausdorff groups is abelian, and inverse limits exist 
and are exact in this category. 

The Hecke equivariance away from $p$ is clear. To prove that the isomorphism is equivariant for the action of $[N_1tN_1]$, it is enough to show
that it is equivariant for $\mathrm{tr}$ and $t_*$. The equivariance for $\mathrm{tr}$ follows
from part (4) of Lemma~\ref{lem:fin level distribution}. The equivariance for $t_*$ follows from 
parts (1) and (3) of Lemma~\ref{lem:fin level distribution}.  
\end{proof}

\subsection{The universal coefficient isomorphism at infinite level}
For $t\in T^+$, recall that $K_t = K_1\cap t^{-1}K_1t = \overline{N}_tT_0N_1$. We consider $\mathscr{D}(K_t/N_1)$ as a 
$\mathscr{D}(K_t)\otimes_{\Z_p}\mathscr{D}(T_0)$-module, where $K_t$ acts by multiplication on the left and 
$T_0$ acts by multiplication on the right. $\mathscr{D}(T_0)$ is a semi-local ring and is complete with respect to its $J$-adic topology, where we let $J$ denote its Jacobson radical. Define 
\[
\mathscr{C}(K_t/N_1)\defeq \Hom_{\mathscr{D}(T_0)}^{\mathrm{cont}}(\mathscr{D}(K_t/N_1), \mathscr{D}(T_0));
\]
this is a $\mathscr{D}(T_0)\otimes_{\Z_p}\mathscr{D}(K_t)$-module. Here $\mathscr{D}(K_t/N_1)$ carries the inverse limit topology from equation \ref{eq:inverse limit}, and we give $\mathscr{C}(K_t/N_1)$ the $J$-adic topology.

\begin{lemma}\label{lem:duality} We have a natural isomorphism 
\[
\mathrm{Hom}_{\mathscr{D}(T_0)}(\mathscr{C}(K_t/N_1), \mathscr{D}(T_0)) \simeq \mathscr{D}(K_t/N_1). 
\]
and, for all $i>0$, $\mathrm{Ext}^i_{\mathscr{D}(T_0)}(\mathscr{C}(K_t/N_1), \mathscr{D}(T_0)) = 0$. (We note that all $\mathscr{D}(T_0)$-linear homomorphisms are automatically 
continuous for $J$-adic topology.)  
\end{lemma}

\begin{proof} From the definition, we have 
\begin{equation}\label{eq:description of C}
\mathscr{C}(K_t/N_1)\cong
\varprojlim_{j,k}\varinjlim_{i}\mathrm{Hom}_{\mathscr{D}(T_0/T_j)}
\left(\mathscr{D}(K_t/\overline{N}_iT_jN_1),\mathscr{D}(T_0/T_j)\right)\otimes_{\Z_p}\Z/p^k. 
\end{equation}
Since $T_0$ normalizes $\overline{N}_i$, there is an isomorphism of profinite sets with $T_0$-action
\[
\overline{N}_t / \overline{N}_i \times T_0 / T_j \toisom K_t / \overline{N}_i T_j N_1 
\]
where $T_0$ acts trivially on $\overline{N}_t / \overline{N}_i$.
Consequently, as a $\mathscr{D}(T_0)$-module,
\[
\mathscr{D}(K_t/\overline{N}_i T_j N_1) \cong \mathscr{D}(\overline{N}_t / \overline{N}_i) \otimes_{\Z_p} \mathscr{D}(T_0/T_j) \,, 
 \]
where $\mathscr{D}(T_0)$ acts trivially on $\mathscr{D}(\overline{N}_t / \overline{N}_i)$.
In particular, $\mathscr{D}(K_t/\overline{N}_i T_j N_1)$ is a finite free $\mathscr{D}(T_0/T_j)$-module, hence 
it is also reflexive. This, together with the explicit description~\eqref{eq:description of C} 
proves the first claim of the lemma. 

For the second claim, observe that $\mathscr{C}(K_t/N_1)$ is 
the $J$-adic completion
of $\mathscr{D}(T_0)$ of the free $\mathscr{D}(T_0)$-module
\[
\varinjlim_i \varprojlim_{j,k} \Hom_{\mathscr{D}(T_0)}(\mathscr{D}(K_t/\overline{N}_i T_j N_1),\mathscr{D}(T_0/T_j)) \otimes \Z/p^k \,. 
\]
Then by \cite[Tag 06LE]{stacks-project}, $\mathscr{C}(K_t/N_1)$ is a flat $\mathscr{D}(T_0)$-module.
Hence, by \cite[Theorem 1]{jensen}, $\Ext^i_{\mathscr{D}(T_0)}(\mathscr{C}(K_t/N_1),\mathscr{D}(T_0)) = 0$ for $i>0$.
\end{proof}

\begin{cor}\label{cor:inf universal coefficient} We have a natural isomorphism 
\[
\mathrm{RHom}_{\mathscr{D}(T_0)}\(C_{\A,\bullet}(K_t, \mathscr{C}(K_t/N_1)),\mathscr{D}(T_0)\)
\simeq C^{\bullet}_{\A}(K_t,\mathscr{D}(K_t/N_1)). 
\]
\end{cor}
\begin{proof} This follows from combining the isomorphism~\eqref{eq:universal coefficient isom} with
Lemma~\ref{lem:duality}. 
\end{proof}

\subsection{Poincar\'e duality}
We start by recalling Poincar\'e duality between Borel--Moore homology and cohomology; recall that $X_{K}^{\BS}$ is homeomorphic
to a smooth compact orientable manifold with boundary. 

\begin{lemma} \label{dual decompositions}
Let $K \subseteq G(\AA_f)$ be a compact open subgroup.
There are homotopy equivalences of complexes of $\Z[K]$-modules
\[ F_{\bullet}^{\BM}[D] \simeq \Hom_{\Z[K]}(F_{\bullet},\Z[K]) \]
\[ F_{\bullet}[D] \simeq \Hom_{\Z[K]}(F_{\bullet}^{\BM},\Z[K]) \,. \]
On the right hand sides, $F_\bullet^{(\BM)}$ is viewed as a left $K$-module by inverting the natural right $K$-module structure.
\end{lemma}

\begin{proof}
We sketch a proof of the second quasi-isomorphism, the proof of the first is completely analogous. As mentioned in Remark~\ref{rem:fin complexes},
the complex $F_{\bullet}$ comes from a combinatorial triangulation
$T$ of $X_K^{\BS}$, which can even be chosen such that $\partial X_K^{\BS}$ is a simplicial subcomplex.
For such a $T$, we can construct the dual $T^\vee$ of this triangulation (this construction and the remaining assertions in this paragraph seem to be well known in topology, see for example~\cite[Ch.~14]{turaev}), which is a CW decomposition of $X_K^{\BS}$. Let $S_\bullet(T)$ denote the simplical homology complex attached to $T$ and let $S_\bullet(T^\vee, T^\vee_{\partial})$ denote the relative cellular homology complex of $T^\vee$ with respect to the boundary $T^\vee_{\partial}$ (which is a sub-CW complex). There are natural perfect pairings
\[
S_{D-k}(T) \times S_{k}(T^\vee, T^\vee_{\partial}) \to \Z
\]
where a simplex pairs to $1$ with its dual cell and to $0$ with the other cells. The induced map $S_\bullet(T)[D] \to \Hom_\Z(S_{\bullet}(T^\vee, T^\vee_{\partial}),\Z)$ is an isomorphism of complexes.

\medskip

Now pull back $T^\vee$ to a $K$-equivariant CW decomposition of $\frakX^{\BS}$, and let $F_\bullet^\vee$ denote the corresponding relative cellular homology complex with respect to the boundary.  Note that the pullback of $T^\vee$ is the dual cell decomposition of the pullback of $T$. Define a pairing
\[
F_{D-k} \times F_{k}^\vee \to \Z[K]
\]
by declaring that a simplex $\sigma$ in $F_{D-k}$ pair to $0$ with a cell $\tau$ in $F_{k}^\vee$ unless $\tau = \sigma^\vee k$ for some $k\in K$, in which case we pair them to $k^{-1}$. From the definitions this pairing induces a map $F_{D-k} \to \Hom_{\Z[K]}(F_{k}^\vee, \Z[K])$, where we view $F_\bullet^\vee$ as a left $K$-module by inverting the natural right $K$-module structure. Using that $S_\bullet(T)[D] \to \Hom_\Z(S_{\bullet}(T^\vee, T^\vee_{\partial}),\Z)$ is an isomorphism, one sees that the maps above form a chain isomorphism $F_\bullet[D] \to \Hom_{\Z[K]}(F_{\bullet}^\vee,\Z[K])$ of complexes of right $K$-modules. The proof is then finished by noting that $F_\bullet^\vee$ is chain homotopic to $F_\bullet^{\BM}$, since they come from $K$-equivariant CW decompositions of the same manifold with boundary.
\end{proof}

\begin{cor}\label{cor:poincare duality}
For any compact open $K \subset G(\A_f)$ and any left $K$-module $M$ (which we also view as a
right $K$-module by inverting the left $K$-module structure), there is a natural quasi-isomorphism
\[ C_{\A,\bullet}^{\BM}(K,M)[D] \simeq C_{\A}^{\bullet}(K,M) \,. \]
\end{cor}

\begin{proof}
Recall that we have homotopy equivalences 
\[ C_{\A,\bullet}^{(\BM)}(K,M) \simeq F^{(\BM)}_{\bullet} \otimes_{\Z[K]} M. \]
Then the result follows from Lemma
\ref{dual decompositions} along with the fact that $F_{\bullet}$,
$F_{\bullet}^{\BM}$ are complexes of finite free $\Z[K]$-modules and, for any finite free $\Z[K]$-module $F$ and any
$\Z[K]$-module $N$,
$\Hom_{\Z[K]}(F,N) \cong \Hom_{\Z[K]}(F,\Z[K]) \otimes_{\Z[K]} N$.
\end{proof}

Let $w\in G(\Z_p)$ be a representative of the longest element of the Weyl group. 
We have an involution $t\mapsto w^{-1}t^{-1}w$ of $T^+$. Let $\tau: K_1\to T_0$ 
be the map that sends $\bar{n}tn\mapsto t$ for $\bar{n}\in \overline{N}_1$,
$t\in T_0$ and $n\in N_1$. Note that $\tau$ is not a homomorphism, but it satisfies $\tau(\ol{b}kb) = \tau(\ol{b})\tau(k)\tau(b)$ for any $\ol{b}\in \ol{N}_1T_0$, $k\in K_1$ and $b\in T_0N_1$. We consider the pairing
\[
\mathscr{D}(K_1/N_1)\times \mathscr{D}(K_1/N_1)\to 
\mathscr{D}(T_0), \langle k_1N_1,k_2N_2\rangle\defeq \tau(wk_2^{-1}w^{-1}k_1), \forall k_1,k_2\in K_1
\]
where the $K_1$ acts on the LHS by left multiplication and on the RHS
by left multiplication by the inverse. This induces a morphism of 
left $K_1$-modules 
\[
\kappa: \mathscr{D}(K_1/N_1)\to (w^{-1})_*\mathscr{C}(K_1/N_1),
\] 
where $(w^{-1})_*\mathscr{C}(K_1/N_1)$ denotes $\mathscr{C}(K_1/N_1)$
with the $K_1$ action twisted such that $k\in K_1$ acts by $wkw^{-1}$.
In turn, $\kappa$ induces a morphism of complexes
\[
\kappa_*: C_{\A,\bullet}(K_1, \mathscr{D}(K_1/N_1))\to C_{\A,\bullet}(K_1, (w^{-1})_*\mathscr{C}(K_1/N_1))\simeq 
C_{\A,\bullet}(K_1, \mathscr{C}(K_1/N_1))
\]

We now consider the composition of morphisms
\[
\pi_w:C^{\BM}_{\A,\bullet}(K_1, \mathscr{D}(K_1/N_1))[D]\toisom  C^{\bullet}_{\A}(K_1,\mathscr{D}(K_1/N_1)) \toisom 
\]
\[
\mathrm{RHom}_{\mathscr{D}(T_0)}\left(C_{\A,\bullet}(K_1, \mathscr{C}(K_1/N_1)),\mathscr{D}(T_0)\right) \to
\]
\[
\mathrm{RHom}_{\mathscr{D}(T_0)}\left(C_{\A,\bullet}(K_1, \mathscr{D}(K_1/N_1)),\mathscr{D}(T_0)\right)
\]
The first morphism is the Poincar\'e duality isomorphism of Corollary~\ref{cor:poincare duality}, the second is the isomorphism in Corollary~\ref{cor:inf universal coefficient}
in the case $t=1$, and the third is the morphism induced by precomposition with $\kappa_*$. 

\begin{prop}\label{prop:Hecke action}
For any $t\in T^+$, the morphism 
\[
\pi_w: C^{\BM}_{\A,\bullet}(K_1, \mathscr{D}(K_1/N_1))[D]\to \mathrm{RHom}_{\mathscr{D}(T_0)}\left(C_{\A,\bullet}(K_1, \mathscr{D}(K_1/N_1)),\mathscr{D}(T_0)\right)
\] 
is equivariant for the Hecke action of $[N_1tN_1]$ on 
$\mathrm{RHom}_{\mathscr{D}(T_0)}\left(C_{\A,\bullet}(K_1, \mathscr{D}(K_1/N_1)),\mathscr{D}(T_0)\right)$ and
the Hecke action of $[N_1 w^{-1}t^{-1}w N_1]$ on 
$C^{\BM}_{\A,\bullet}(K_1, \mathscr{D}(K_1/N_1))[D]$. 

The morphism $\pi_w$ is also equivariant for the Hecke action of $[K^p(g^p)^{-1}K^p]$ on 
the LHS and the Hecke action of  $[K^pg^pK^p]$ on the RHS. 
\end{prop}

\begin{proof} By translating the Hecke action of $[N_1 w^{-1}t^{-1}w N_1]$
on $C^{\BM}_{\A,\bullet}(K_1, \mathscr{D}(K_1/N_1))[D]$, as described in Theorem~\ref{thm:completed hom via distributions}, 
under the inverse of the isomorphism 
\[
\mathrm{RHom}_{\mathscr{D}(T_0)}\left(C_{\A,\bullet}(K_1, \mathscr{C}(K_1/N_1)),\mathscr{D}(T_0)\right)
\toisom C^{\BM}_{\A,\bullet}(K_1, \mathscr{D}(K_1/N_1))[D],
\]
we obtain the following composition: 
\[
C_{\A,\bullet}(K_1, \mathscr{C}(K_1/N_1))\stackrel{\eta_{K_1,K_t}}{\toisom} 
C_{\A,\bullet}(K_t, \mathscr{C}(K_t/N_1))\stackrel{\iota_{K'_t,K_1}\circ(w^{-1}tw)_*}{\to}
C_{\A,\bullet}(K_1, \mathscr{C}(K_1/N_t))
\]
\[
\stackrel{\mathrm{tr}}{\to} C_{\A,\bullet}(K_1, \mathscr{C}(K_1/N_1)). 
\]
We have to check that the Hecke action of $[N_1tN_1]$ on
$C_{\A,\bullet}(K_1, \mathscr{D}(K_1/N_1))$ corresponds to this 
composition under the morphism $C_{\A,\bullet}(K_1, \mathscr{D}(K_1/N_1))\to 
C_{\A,\bullet}(K_1, \mathscr{C}(K_1/N_1))$ induced by $\kappa_w$. 
Again, using the explicit description of the Hecke action in 
Theorem~\ref{thm:completed hom via distributions},
we obtain the composition 
\[
C_{\A,\bullet}(K_1, \mathscr{D}(K_1/N_1))\stackrel{\mathrm{tr}}{\to}
C_{\A,\bullet}(K_1, \mathscr{D}(K_1/N_t))\stackrel{t_*\circ \eta_{K_1,K'_t}}{\to}
C_{\A,\bullet}(K_t, \mathscr{D}(K_t/N_1))
\]
\[
\stackrel{\iota_{K_t,K_1}}{\toisom} C_{\A,\bullet}(K_1, \mathscr{D}(K_1/N_1)). 
\]

We now define a morphism $\kappa_{t}: \mathscr{D}(K_t/N_1)\to ((t^{-1}w)^{-1})_* \mathscr{C}(K_t/N_1)$
from the pairing 
\[
\mathscr{D}(K_t/N_1)\times \mathscr{D}(K_t/N_1)\to \mathscr{D}(T_0), \langle k_1N_1,k_2N_1 \rangle\mapsto \tau(wk_2^{-1}w^{-1}tk_1t^{-1}), 
\]
where note that $wk_2^{-1}w^{-1}tk_1t^{-1}\in tK_tt^{-1}\subseteq K_1$. 
We obtain an induced map 
\[
\kappa_{t*}:C_{\A,\bullet}(K_t, \mathscr{D}(K_t/N_1))\to C_{\A,\bullet}(K_t, \mathscr{C}(K_t/N_1)).
\]
The Hecke equivariance at $p$ now follows from Lemmas~\ref{lem:commutative diagram 1} and~\ref{lem:commutative diagram 2}
below. 

The Hecke equivariance away from $p$ is clear, taking into account that the Poincar\'e duality
isomorphism matches $[K^p(g^p)^{-1}K^p]$ on $C_{\A,\bullet}^{\BM}(K_1,\mathscr{D}(K_1/N_1))[D]$ 
with $[K^pg^pK^p]$ on $C_{\A,\bullet}(K_1,\mathscr{D}(K_1/N_1))$, and that the other morphisms are
Hecke equivariant away from $p$. 
\end{proof}

\begin{lemma}\label{lem:commutative diagram 1}The following diagram is commutative:
\[
\xymatrix{C_{\A,\bullet}(K_1, \mathscr{D}(K_1/N_1))\ar[r]^{\mathrm{tr}}\ar[d]^{\kappa_{*}}
& C_{\A,\bullet}(K_1, \mathscr{D}(K_1/N_t))\ar[r]^{t_*\circ \eta_{K_1,K'_t}}
& C_{\A,\bullet}(K_t, \mathscr{D}(K_t/N_1))\ar[d]^{\kappa_{t*}} \\
C_{\A,\bullet}(K_1,\mathscr{C}(K_1/N_1))\ar[rr]^{\eta_{K_1,K_t}} &\ & C_{\A,\bullet}(K_t, \mathscr{C}(K_t/N_1))}.
\]
\end{lemma}

\begin{proof} We have the Cartesian diagram 
\begin{equation}\label{eq:Cartesian diagram}
\xymatrix{X_{K_t}\ar[d]^{t^{-1}}\ar[r]^{t^{-1}w} & X_{K_t}\ar[d]^{1} \\
X_{K_1}\ar[r]^{w} & X_{K_1}},
\end{equation}
where the horizontal arrows are isomorphisms, the left vertical arrow is right multiplication by $t^{-1}$ followed
by the natural projection, and the right vertical arrow is the natural projection. Using the maps in this diagram,
we obtain the following diagram of local systems on $X_{K_1}$:
\begin{equation}\label{eq:loc sys diagram}
\xymatrix{\mathscr{D}(K_1/N_1)\ar[r]^{\mathrm{tr}}\ar[d]^{\kappa} & 
\mathscr{D}(K_1/N_t)\ar[r]^{t_*} & (t^{-1})_*\mathscr{D}(K_t/N_1)\ar[d]^{(t^{-1})_*(\kappa_t)} \\
(w^{-1})_*\mathscr{C}(K_1/N_1)\ar[r]^{\sim} & (w^{-1})_*1_*\mathscr{C}(K_t/N_1)\ar[r]^{\sim}& (t^{-1})_*((t^{-1}w)^{-1})_*\mathscr{C}(K_t/N_1)
}.
\end{equation}
Note that $(t^{-1})_*((t^{-1}w)^{-1})_* = (w^{-1})_*1_*$ by~\eqref{eq:Cartesian diagram}. 
We will prove that this diagram commutes. 

To see this, we claim that the map $\mathscr{D}(K_1/N_t)\to (w^{-1})_*\mathscr{C}(K_1/N_1)$ 
obtained by going clockwise around~\eqref{eq:loc sys diagram} is induced from the pairing
\[
\mathscr{D}(K_1/N_t)\times (w^{-1})_*\mathscr{D}(K_1/N_1)
\]
given by
\begin{equation}\label{eq:pairing}
\langle k_1N_t, k_2N_1\rangle = \begin{cases}\tau(wk_2^{-1}w^{-1}k_1)
\ \mathrm{if}\ wk_2^{-1}w^{-1}k_1\in \overline{N}_1T_0N_t \\
0\ \mathrm{otherwise}\end{cases}. 
\end{equation}
Indeed, the map $t_*: \mathscr{D}(K_1/N_t)\to (t^{-1})_*\mathscr{D}(K_t/N_1)$ factors as 
\[
\mathscr{D}(K_1/N_t)\toisom \mathrm{Ind}_{K'_t}^{K_1} \mathscr{D}(K'_t/N_t)\toisom (t^{-1})_*\mathscr{D}(K_t/N_1),
\]
where the first map is given by $\mu = \sum_{s\in K_1/K'_t} s\mu_s \mapsto (\mu_s)_{s\in K_1/K'_t}$ and the second
map is given by $(\mu_s)_s\mapsto (t^{-1}\mu_s t)_s$. Similarly, the map $\mathscr{C}(K_1/N_1)\toisom 1_*\mathscr{C}(K_t/N_1)$
identifies $\mathscr{C}(K_1/N_1)$ with $\mathrm{Ind}^{K_1}_{K_t}\mathscr{C}(K_t/N_1)$. 
Recall that $K'_t = \overline{N}_1T_0N_t= wK_tw^{-1}$. The map 
\[
\mathrm{Ind}_{K'_t}^{K_1} \mathscr{D}(K'_t/N_t) \to (w^{-1})_*\mathrm{Ind}^{K_1}_{K_t}\mathscr{C}(K_t/N_1)
\] 
obtained by going clockwise around~\eqref{eq:loc sys diagram} 
is simply induced from the map $\mathscr{D}(K'_t/N_t)\to (w^{-1})_*\mathscr{C}(K_t/N_1)$ 
determined by the pairing
\[
\langle k_1N_t,k_2N_1\rangle = \tau(wk_2^{-1}w^{-1}k_1). 
\]
The claim now follows from the definition of the two inductions. 

Now that this claim is established, we notice that any left coset $k_1N_1$ with $k_1\in K_1$ gets sent to  
$[N_1:N_t]$ left $N_t$-cosets under $\mathrm{tr}$, 
and precisely one of these pairs non-trivially with $k_2N_1$ with $k_2\in K_1$ under the 
pairing in~\eqref{eq:pairing}.
This proves that the diagram~\eqref{eq:loc sys diagram} commutes and therefore proves the lemma.  
\end{proof}

\begin{lemma}\label{lem:commutative diagram 2} The following diagram is commutative:
\[
\xymatrix{C_{\A,\bullet}(K_t,\mathscr{D}(K_t/N_1))\ar[rr]^{\iota_{K_t,K_1}}\ar[d]^{\kappa_{t*}} &\ & 
C_{\A,\bullet}(K_1, \mathscr{D}(K_t/N_1)\ar[d]^{\kappa_{*}}) \\
C_{\A,\bullet}(K_t, \mathscr{C}(K_t/N_1))\ar[r]^{\iota_{K'_t,K_1}\circ(w^{-1}tw)_*} &
C_{\A,\bullet}(K_1, \mathscr{C}(K_1/N_t))\ar[r]^{\mathrm{tr}} &
C_{\A,\bullet}(K_1, \mathscr{C}(K_1/N_1)).
}
\]
\end{lemma}

\begin{proof} The proof is analogous to that of Lemma~\ref{lem:commutative diagram 1}. 
\end{proof}

\subsection{Ordinary parts}\label{subsec:ordinary parts} The goal of this section is to show that 
the morphism $\pi_w$ in Proposition~\ref{prop:Hecke action} induces an isomorphism
on ordinary parts. 

We start by defining the ordinary part of homology at finite level, through the means of a projector. We focus on what we need; the interested reader may consult \cite[\S 2]{khare-thorne} for an abstract viewpoint. Let $j\in \Z_{\geq 1}$ and $t\in T^+$. Assume that the compact open subgroup $\overline{N}_t T_j N_1\subseteq K_1$ admits an 
Iwahori factorisation; this can always be ensured by choosing $t$ large enough with respect to $j$. For any $s\in T^+$, we let $U_s$ denote the double coset operator 
$[\overline{N}_tT_jN_1s\overline{N}_tT_jN_1]$ acting on $C^{(\BM)}_{\A,\bullet}(\overline{N}_tT_jN_1, \Z_p)\simeq C^{(\BM)}_{\A,\bullet}(K_1, \mathscr{D}(K_1/\overline{N}_tT_jN_1))$. 

\begin{lemma}\label{lem:basic properties}\leavevmode
\begin{enumerate}
\item For $s_1, s_2\in T^+$ we have $U_{s_1s_2} = U_{s_1}\circ U_{s_2}$. In particular, all these operators commute. 
\item For $t'\geq t$, $j'\geq j$, the restriction morphism 
\[
\mathrm{res}: C^{(\BM)}_{\A,\bullet}(K_1, \mathscr{D}(K_1/\overline{N}_{t'}T_{j'}N_1))\to 
C^{(\BM)}_{\A,\bullet}(K_1, \mathscr{D}(K_1/\overline{N}_tT_jN_1))
\]
is $U_s$-equivariant for all $s\in T^+$.  
\item For any $s\in T^+$, we have a commutative diagram
\[
\xymatrix{C^{(\BM)}_{\A,\bullet}(K_1, \mathscr{D}(K_1/\overline{N}_{ts}T_jN_1))\ar[d]^{U_s}\ar[r]^{\mathrm{res}} & 
C^{(\BM)}_{\A,\bullet}(K_1, \mathscr{D}(K_1/\overline{N}_{t}T_jN_1))\ar[d]^{U_s}\ar[ld]^{s_*\circ \mathrm{tr}} \\
C^{(\BM)}_{\A,\bullet}(K_1, \mathscr{D}(K_1/\overline{N}_{ts}T_jN_1))\ar[r]^{\mathrm{res}} & C^{(\BM)}_{\A,\bullet}(K_1, \mathscr{D}(K_1/\overline{N}_{t}T_jN_1))}.
\]
\end{enumerate}
\end{lemma}

\begin{proof} The first part follows from the explicit description of the Hecke operators $U_{s_1}$, $U_{s_2}$,
and $U_{s_1s_2}$, cf.~\eqref{eq:hecke fin homology} and from the same computation as in~\cite[Lemma 3.1.4]{emerton-ord1}.  
The second part again follows from the explicit description in~\eqref{eq:hecke fin homology} and from the fact that the Iwahori 
factorisation gives a bijection $\overline{N}_tT_jN_1/\overline{N}_tT_jN_s\toisom N_1/N_s$, which shows that the coset representatives
can be chosen independently of $j$ and $t$. We now prove the third part. The commutativity of the lower triangle follows from
the definition of $U_s$. For the upper triangle, we write the definition of $U_s$ acting on 
$C_{\A,\bullet}(K_1, \mathscr{D}(K_1/\overline{N}_{ts}T_jN_1))\simeq C_{\A,\bullet}(\overline{N}_{ts}T_jN_1, \Z_p)$.
By definition, this is equal to 
\[
C_{\A,\bullet}(\overline{N}_{ts}T_jN_1, \Z_p)\stackrel{\mathrm{tr}}{\to} C_{\A,\bullet}(\overline{N}_{ts}T_jN_s, \Z_p)\stackrel{s_*}{\to} C_{\A,\bullet}(\overline{N}_{ts^2}T_jN_1, \Z_p)
\stackrel{\mathrm{res}}{\to} C_{\A,\bullet}(\overline{N}_{ts}T_jN_1, \Z_p),
\]
which can be rewritten as 
\begin{equation}\label{eq:rewritten Hecke}
C_{\A,\bullet}(\overline{N}_{ts}T_jN_1, \Z_p)\stackrel{\mathrm{tr}}{\to} C_{\A,\bullet}(\overline{N}_{ts}T_jN_s, \Z_p)\stackrel{\mathrm{res}}{\to} C_{\A,\bullet}(\overline{N}_{t}T_jN_{s}, \Z_p)
\stackrel{s_*}{\to} C_{\A,\bullet}(\overline{N}_tT_jN_1, \Z_p).
\end{equation}
We claim that the diagram of locally symmetric spaces with natural projection morphisms
\[
\xymatrix{X_{\overline{N}_{ts}T_jN_s}\ar[d]\ar[r] & X_{\overline{N}_{t}T_jN_s}\ar[d] \\
X_{\overline{N}_{ts}T_jN_1}\ar[r] & X_{\overline{N}_{t}T_jN_1}}
\]
is Cartesian. To see this, it is enough to see that the morphism
\[
X_{\overline{N}_{ts}T_jN_s}\to X_{\overline{N}_{ts}T_jN_1}\times_{X_{\overline{N}_{t}T_jN_1}}X_{\overline{N}_{t}T_jN_s}
\] 
induces an isomorphism of the fibers over $X_{\overline{N}_{ts}T_jN_1}$. Indeed, the fibers on the 
RHS can be identified with $\overline{N}_{ts}T_jN_1/ \overline{N}_{ts}T_jN_s$ and the fibers on the LHS can be 
identified with $\overline{N}_{t}T_jN_1/ \overline{N}_{t}T_jN_s$, cf~\cite[Lemma 6.2.1]{arizona}. 
The claim now follows from the Iwahori factorisation, since all the fibers are identified with $N_1/N_s$. 
Using proper base change, we rewrite~\eqref{eq:rewritten Hecke} as 
\[
C_{\A,\bullet}(\overline{N}_{ts}T_jN_1, \Z_p)\stackrel{\mathrm{res}}{\to} C_{\A,\bullet}(\overline{N}_{t}T_jN_1, \Z_p)\stackrel{\mathrm{tr}}{\to} C_{\A,\bullet}(\overline{N}_{t}T_jN_{s}, \Z_p)
\stackrel{s_*}{\to} C_{\A,\bullet}(\overline{N}_{ts}T_jN_1, \Z_p).
\]
This completes the proof. 
\end{proof}

We call $s\in T^+$ a \emph{controlling element} if $\cap_{i\geq 1} s^iN_1s^{-i} = \{1\}$. Such a controlling element
always exists, for example by taking $s = \prod_{\alpha} p^{\alpha}$, where the product runs over the positive coroots
of $G$. 

\begin{lemma}\label{lem:controlling element} Let $s\in T^+$ be a controlling element. 
For any $s_1\in T^+$, there exists $i\in \Z_{\geq 1}$ such that $s_1s_2 = s^i$ for some $s_2\in T^+$. 
\end{lemma}

\begin{proof} Since $\cap_{i\geq 1} s^iN_1s^{-i} = \{1\}$, there exists $i\in \Z_{\geq 1}$ such that 
$s^i N_1 s^{-i}\subseteq s_1 N_1 s_1^{-1}$. This shows that $s_2:=s_1^{-1} s^i \in T^+$. 
\end{proof}

Let $s_0$ be a controlling element. Using homotopy equivalences, we transport 
$U_{s_0}$ on $C^{(\BM)}_{\A,\bullet}(K_1, \mathscr{D}(K_1/\overline{N}_tT_jN_1)\otimes \Z/p^k)$ to an operator 
$\widetilde{U}_{s_0}$ on the corresponding Borel--Serre complex $C^{(\BM)}_{\bullet}(K_1, \mathscr{D}(K_1/\overline{N}_tT_jN_1)\otimes \Z/p^k)$,
which acts as $U_{s_0}$ up to homotopy and, in particular, induces the same action on homology. 
Since the latter is a complex of finite projective $\Z/p^k$-modules, $\widetilde{U}_{s_0}^{N!}$ stabilises 
to an idempotent. We denote the corresponding direct summand by
$C^{(\BM)}_{\bullet}(K_1, \mathscr{D}(K_1/\overline{N}_tT_jN_1)\otimes \Z/p^k)^{T^+-\mathrm{ord}}$.
We define 
\[
C^{(\BM)}_{\bullet}(K_1, \mathscr{D}(K_1/\overline{N}_tT_jN_1))^{T^+-\mathrm{ord}}\defeq 
\varprojlim_{k}C^{(\BM)}_{\bullet}(K_1, \mathscr{D}(K_1/\overline{N}_tT_jN_1)\otimes \Z/p^k)^{T^+-\mathrm{ord}}. 
\]
By~\cite[Lemma 2.13]{khare-thorne}, this is a direct summand of $C^{(\BM)}_{\bullet}(K_1, \mathscr{D}(K_1/\overline{N}_tT_jN_1))$
whose homology recovers the ordinary part of homology / Borel--Moore homology with respect to $s_0$. 

Note that the homotopy equivalences between the adelic and the Borel--Serre complexes are functorial in the coefficients. Therefore, the formation of the direct summands is compatible with the transition morphisms between different levels, and
we can also define the ordinary part 
\[
C_{\bullet}^{(\BM)}(K_1, \mathscr{D}(K_1/N_1))^{T^+-\mathrm{ord}}:= \varprojlim_{j,t}C^{(\BM)}_{\bullet}(K_1, \mathscr{D}(K_1/\overline{N}_tT_jN_1))^{T^+-\mathrm{ord}}
\]
of $C_{\bullet}^{(\BM)}(K_1, \mathscr{D}(K_1/N_1))$ with respect to $s_0$.  
This is a direct summand of $C_{\bullet}^{(\BM)}(K_1, \mathscr{D}(K_1/N_1))$ and its homology recovers 
$H_*^{(\BM)}(N_1)^{T^+-\mathrm{ord}}$ by~\cite[Lemma 2.13]{khare-thorne} and Theorem \ref{thm:completed hom via distributions}. 

Using Lemma~\ref{lem:controlling element},
one can check that all this is independent of the choice of controlling element $s_0$. By Lemma~\ref{lem:controlling element} and part (1) of Lemma~\ref{lem:basic properties}, 
we see that $U_{s}$ acts as a quasi-isomorphism on $C^{(\BM)}_{\bullet}(K_1, \mathscr{D}(K_1/\overline{N}_tT_jN_1))^{T^+-\mathrm{ord}}$ for any $s\in T^+$.  
As a result, we obtain the following horizontal control theorem. 

\begin{prop}\label{prop:control for D} For any $s\in T^+$, the transition morphisms 
\[
C^{(\BM)}_{\bullet}(K_1, \mathscr{D}(K_1/\overline{N}_{ts} T_j N_1))^{T^+-\mathrm{ord}}\to C^{(\BM)}_{\bullet}(K_1, \mathscr{D}(K_1/\overline{N}_t T_j N_1))^{T^+-\mathrm{ord}},
\]
are quasi-isomorphisms. 
\end{prop}

\begin{proof} It is enough to show that the transition morphisms induce an isomorphism on the ordinary part of homology / Borel--Moore homology. 
Surjectivity follows from the commutativity of the lower triangle in part (3) of Lemma~\ref{lem:basic properties}, since $U_s$ acts as an isomorphism
on homology. Injectivity follows from the commutativity of the upper triangle in part (3) of Lemma~\ref{lem:basic properties}. 
\end{proof}

Assume again that $\overline{N}_tT_jN_1$ admits an Iwahori factorisation. Set 
\[
\mathscr{C}(K_1/\overline{N}_tT_jN_1) \defeq 
\mathrm{Hom}_{\mathscr{D}(T_0)}\left(\mathscr{D}(K_1/\overline{N}_tT_jN_1), \mathscr{D}(T_0/T_j)\right).
\]
We also have an action of $s\in T^-$ on 
each $C_{\A,\bullet}(K_1, \mathscr{C}(K_1/\overline{N}_tT_jN_1))$ via a double coset operator $U_s$. We define the 
analogous notion of controlling element and use it to define the ordinary part $C_{\bullet}(K_1, \mathscr{C}(K_1/\overline{N}_tT_jN_1))^{T^--\mathrm{ord}}$
with respect to $T^-$. In this setting, we have the following horizontal control theorem. 

\begin{prop}\label{prop:control for C} For any $s\in T^-$, the transition morphisms 
\[
C_{\bullet}(K_1, \mathscr{C}(K_1/\overline{N}_{t} T_j N_1))^{T^--\mathrm{ord}}\to C_{\bullet}(K_1, \mathscr{C}(K_1/\overline{N}_{ts} T_j N_1))^{T^--\mathrm{ord}},
\]
are quasi-isomorphisms. 
\end{prop}

\begin{proof} This is proved analogously to Proposition~\ref{prop:control for D}.
\end{proof}

Set $\mathscr{C}(K_1/\overline{N}_tT_jN_1, \Z/p^k)\defeq \mathscr{C}(K_1/\overline{N}_tT_jN_1)\otimes_{\Z}\Z/p^k$. We have 
\[
\mathscr{C}(K_1/N_1)\simeq \varprojlim_{j,k}\varinjlim_{t} \mathscr{C}(K_1/\overline{N}_tT_jN_1, \Z/p^k),
\]
where the inverse limit runs over $j,k\in \mathbb{Z}_{\geq 1}$ and the direct limit runs over those $t\in T^+$ that 
are sufficiently large with respect to $j$. We define 
\[
C_{\bullet}(K_1, \mathscr{C}(K_1/N_1))^{T^--\mathrm{ord}}\defeq \varprojlim_{j,k}\varinjlim_{t} C_{\bullet}(K_1, \mathscr{C}(K_1/\overline{N}_{t} T_j N_1,\Z/p^k))^{T^--\mathrm{ord}}, 
\]
the ordinary part of $C_{\bullet}(K_1, \mathscr{C}(K_1/N_1))$ with respect to $T^-$.

Also set
\[
(w^{-1})_*\mathscr{C}(K_1/\overline{N}_tT_jN_1) \defeq 
\mathrm{Hom}_{\mathscr{D}(T_0)}\left((w^{-1})_*\mathscr{D}(K_1/\overline{N}_tT_jN_1), \mathscr{D}(T_0/T_j)\right).
\]
For each  $j\in \mathbb{Z}_{\geq 1}$, the map $\kappa: \mathscr{D}(K_1/N_1)\to (w^{-1})_*\mathscr{C}(K_1/N_1)$
induces a finite level map 
\[
\kappa_j: \mathscr{D}(K_1/\overline{N}_tT_jN_1)\to (w^{-1})_*\mathscr{C}(K_1/\overline{N}_tT_jN_1), 
\]
whenever $t\in T^+$ is such that $\overline{N}_tT_jN_1$ admits an Iwahori factorisation.  

\begin{lemma}\label{lem:fin level iso} The morphism
\[
(\kappa_j)_*: C_{\A,\bullet}(K_1,\mathscr{D}(K_1/\overline{N}_t T_j N_1))\to C_{\A,\bullet}(K_1,\mathscr{C}(K_1/\overline{N}_tT_jN_1))
\]
induces a quasi-isomorphism 
\[
(\kappa_j)_*^{\mathrm{ord}}: C_{\bullet}(K_1,\mathscr{D}(K_1/\overline{N}_t T_j N_1))^{T^+-\mathrm{ord}}\toisom 
C_{\bullet}(K_1,\mathscr{C}(K_1/\overline{N}_tT_jN_1))^{T^--\mathrm{ord}}.
\]
\end{lemma}

\begin{proof} We have the following finite-level version of~\eqref{eq:loc sys diagram}:
\begin{equation}\label{eq:fin loc sys diagram}
\xymatrix{\mathscr{D}(K_1/\overline{N}_t T_jN_1)\ar[r]\ar[d]^{\kappa_j} & 
\mathscr{D}(K_1/\overline{N}_1T_jN_t)\ar[r]^{t_*\ \sim} & (t^{-1})_*\mathscr{D}(K_t/\overline{N}_t T_jN_1)\ar[d]^{(t^{-1})_*(\kappa_{j,t})} \\
(w^{-1})_*\mathscr{C}(K_1/\overline{N}_t T_jN_1)\ar[r]^{\sim} & (w^{-1})_*1_*\mathscr{C}(K_t/\overline{N}_t T_jN_1)
\ar[r]^{\sim}& (t^{-1})_*((t^{-1}w)^{-1})_*\mathscr{C}(K_t/\overline{N}_t T_jN_1)
},
\end{equation}
where the top row is part of the definition of the double coset operator corresponding to $t$. More 
precisely, in order to get the double coset operator $U_t$ 
acting on $C_{\A,\bullet}(K_1, \mathscr{D}(K_1/\overline{N}_t T_jN_1))$, 
one needs to apply $C_{\A,\bullet}(K_1,\ )$ to the top row and compose with the isomorphism 
\[
\iota_{K_t,K_1}: C_{\A,\bullet}(K_t, \mathscr{D}(K_t/\overline{N}_t T_jN_1))\toisom C_{\A,\bullet}(K_1, \mathscr{D}(K_1/\overline{N}_t T_jN_1)).
\]

We claim that the right vertical arrow in~\eqref{eq:fin loc sys diagram} is an isomorphism. For this, observe that the natural map $K_t/\overline{N}_t T_j N_1\to T_0/T_j$
is a bijection, so the pairing 
\[
\mathscr{D}(K_t/\overline{N}_t T_jN_1)\times ((t^{-1}w)^{-1})_* \mathscr{D}(K_t/\overline{N}_t T_jN_1) \to \mathscr{D}(T_0/T_j)
\] 
\[
\langle t_1 \overline{N}_tT_jN_1, t_2\overline{N}_tT_jN_1 \rangle = wt_2^{-1}w^{-1}tt_1t^{-1} T_j = wt_2^{-1}w^{-1}t_1 T_j
\]
is perfect. All of this implies that $(\kappa_j)_*$ is the composite of the double coset operator $U_t$ acting on 
$C_{\A,\bullet}(K_1, \mathscr{D}(K_1/\overline{N}_t T_jN_1))$ with an isomorphism. 

In order to prove the lemma, it suffices to check that $(\kappa_j)_*$ induces an isomorphism 
on the ordinary part of homology. The map induced by $(\kappa_j)_*$ on homology induces a morphism of ordinary parts of homology (where the ordinary
part is taken with respect to $T^+$ on the RHS and with respect to $T^-$ on the LHS), 
which factors as $U_t$ composed with an isomorphism. 
Since $U_t$ acts as an isomorphism on the ordinary part of homology, the lemma follows. 
\end{proof}

\begin{prop}\label{prop:ordinary iso} The map 
\[
\kappa_*: C_{\A,\bullet}(K_1,\mathscr{D}(K_1/N_1))\to C_{\A,\bullet}(K_1,\mathscr{C}(K_1/N_1))
\]
induces a quasi-isomorphism
\[
\kappa^{\mathrm{ord}}_*: C_{\bullet}(K_1,\mathscr{D}(K_1/N_1))^{T^+-\mathrm{ord}}\toisom C_{\bullet}(K_1,\mathscr{C}(K_1/N_1))^{T^--\mathrm{ord}}.
\]
\end{prop}

\begin{proof}This follows by combining Lemma~\ref{lem:fin level iso} and
Propositions~\ref{prop:control for D} and~\ref{prop:control for C}. 
\end{proof}

\begin{thm}\label{thm:ordinary PD} The morphism $\pi_w$ induces a quasi-isomorphism of ordinary parts
\[
\pi_w^{\mathrm{ord}}: C^{\BM}_{\bullet}(K_1, \mathscr{D}(K_1/N_1))^{T^+-\mathrm{ord}}[D] \to 
\mathrm{RHom}_{\mathscr{D}(T_0)}\left(C_{\bullet}(K_1, \mathscr{D}(K_1/N_1))^{T^+-\mathrm{ord}},\mathscr{D}(T_0)\right). 
\]
\end{thm}

\begin{proof} By Proposition~\ref{prop:ordinary iso}, we have a quasi-isomorphism:
\[
\mathrm{RHom}_{\mathscr{D}(T_0)}\left(C_{\bullet}(K_1, \mathscr{C}(K_1/N_1))^{T^--\mathrm{ord}},\mathscr{D}(T_0)\right)\toisom 
\mathrm{RHom}_{\mathscr{D}(T_0)}\left(C_{\bullet}(K_1, \mathscr{D}(K_1/N_1))^{T^+-\mathrm{ord}},\mathscr{D}(T_0)\right).
\]
For each $k,j\in \Z_{\geq 1}$, Poincar\'e duality and the universal coefficient isomorphism (at each level $\overline{N}_tT_jN_1$) 
induce quasi-isomorphisms 
\[
\varprojlim_{t} C^{\BM}_{\bullet}(K_1, \mathscr{D}(K_1/\overline{N}_tT_jN_1)\otimes \Z/p^k)^{T^+-\mathrm{ord}}[D] \toisom 
\]
\[
\mathrm{RHom}_{\mathscr{D}(T_0/T_j)\otimes \Z/p^k}\left(\varinjlim_{t}C_{\bullet}(K_1, \mathscr{C}(K_1/\overline{N}_t T_jN_1)\otimes \Z/p^k)^{T^--\mathrm{ord}},\mathscr{D}(T_0/T_j)\otimes \Z/p^k\right),
\]
where the transition morphisms in both the inverse and the direct limit are isomorphisms by the horizontal control theorems. 
The morphism is a quasi-isomorphism, since on perfect complexes of $\Z/p^k$-modules, the ordinary part is simply
the largest direct summand on which $U_{s_0}$ acts invertibly, and this commutes with
\[
\mathrm{RHom}_{\mathscr{D}(T_0/T_j)\otimes \Z/p^k}(-, \mathscr{D}(T_0/T_j)\otimes \Z/p^k).
\] 
Taking inverse limits with respect to $j,k$, we obtain a quasi-isomorphism 
\[
C^{\BM}_{\bullet}(K_1, \mathscr{D}(K_1/N_1))^{T^+-\mathrm{ord}}[D]\toisom
\mathrm{RHom}_{\mathscr{D}(T_0)}\left(C_{\bullet}(K_1, \mathscr{C}(K_1/N_1))^{T^--\mathrm{ord}},\mathscr{D}(T_0)\right). 
\]
The theorem follows. 
\end{proof}

\subsection{An application of Theorem~\ref{main complex}}\label{subsec: application}
We will give some implications of Theorem~\ref{main complex} for completed homology and completed Borel--Moore homology groups. In the cases we will be interested in below, $T$ is split over $\Zp$ and hence $T_0 \cong (\Zp^\times)^{\dim T}$. The ring $\mathscr{D}(T_0)$ is then easily seen to be a semi-local complete intersection ring. For any such ring $A$, and a finitely generated $A$-module $M$, one may define
\[
\codim_{A}M = \inf_j \{j \mid \Ext_{A}^j(M,A) \neq 0 \}.
\]
Geometrically, $\codim_{A}M$ is the minimum of the codimensions of the support of $M$ at the maximal ideals of $A$. For a closed subgroup $H \subset G(\Zp)$, we let
\[ \tH_i(H) \colonequals \varprojlim_{n,K \supseteq H} H_i(X_K,\Z/p^n) \,. \]

\begin{thm} \label{homology vanishing}
Assume that $G$ admits a Shimura datum of Hodge type and that $G_{\Qp}$ is
split. Recall that $d=D/2$ is the complex dimension of the Shimura varieties for $G$. 
\begin{enumerate}
\item Let $H \subseteq U(\Zp)$ be a closed subgroup.  Then
\[ \varprojlim_{K \supseteq H} H_i^{\BM}(X_{K},\Z/p^r) = 0 \]
for all $r \ge 1$ and $i>d$.
\item
We have
\[ \codim_{\mathscr{D}(T_0)} \tH_i(N_1)^{T^+-\ord} \ge d-i \]
for all $0 \le i \le d$.
\end{enumerate}
\end{thm}
\begin{remark}
The slightly unusual level $N_1$ was chosen for convenience.  The groups
$\tH_i(N)^{T^+-\ord}$ for $N$ a compact open subgroup of $U(\Zp)$ are
all isomorphic.  For $N' \subseteq N \subseteq U(\Zp)$, the trace map
$\tH_i(N)^{T^+-\ord} \to \tH_i(N')^{T^+-\ord}$ is an isomorphism for the following reason.
We can find $s \in T^+$ so that $sNs^{-1} \subseteq N'$.  Consider the diagram
\[ 
\tH_i(N)^{T^+-\ord} \xrightarrow{\tr_1} \tH_i(N')^{T^+-\ord} \xrightarrow{\tr_2} \tH_i(sNs^{-1})^{T^+-\ord} \xrightarrow{\tr_3}  \tH_i(sN's^{-1})^{T^+-\ord} \,.
\]
The maps $\tr_2 \circ \tr_1 = [NsN]s^{-1}$ and $\tr_3 \circ \tr_2 = [N'sN']s^{-1}$ are isomorphisms.  Then
$(\tr_2 \circ \tr_1)^{-1} \circ \tr_2$ is both a left and right inverse
of $\tr_1$ since
\[ 
\tr_1 \circ (\tr_2 \circ \tr_1)^{-1} \circ \tr_2 = (\tr_3 \circ \tr_2)^{-1} \circ \tr_3 \circ \tr_2 \circ \tr_1 \circ (\tr_2 \circ \tr_1)^{-1} \circ \tr_2 = \id \,. 
\]
\end{remark}
\begin{lemma} \label{universal coefficient}
For each compact open subgroup $K \subseteq G(\Zp)$, there is an
isomorphism of $\Z/p^r$-modules
\[ \Hom_{\ZZ/p^r}\left(H^{\bullet}_c(X_K,\ZZ/p^r),\ZZ/p^r \right) \simeq
H_{\bullet}^{\BM}(X_K,\ZZ/p^r) \,, \]
and these isomorphisms are compatible with changing the level.
\end{lemma}
\begin{proof}
The result can be obtained by applying the functor $\Hom_{\ZZ/p^r}(-,\ZZ/p^r)$
to the universal coefficient isomorphism
\[ H^{\bullet}_c(X_K,\ZZ/p^r) \simeq \Hom_{\ZZ/p^r}\left( H_{\bullet}^{\BM}(X_K,\ZZ/p^r), \ZZ/p^r \right) \,, \]
and using the fact that any finite $\Z/p^r$-module is naturally
isomorphic to its double dual.

The universal coefficent isomorphism can be proved as follows.
We have isomorphisms
\[ H^{\BM}_{*}(X_K,\Z/p^r) \simeq H_*\left(F_{\bullet}^{\BM} \otimes_{\Z[K^p K]} \Z/p^r \right) \]
\[ H^*_c(X_K,\Z/p^r) \simeq H^*\left(\Hom_{\Z[K^p K]} \left(F_{\bullet}^{\BM}, \Z/p^r\right) \right) \]
By the adjunction between tensor products and homomorphisms,
\[ \Hom_{\Z[K^p K]} \left(F_{\bullet}^{\BM}, \Z/p^r\right) \simeq \Hom_{\Z/p^r} \left( F_{\bullet}^{\BM} \otimes_{\Z[K^p K]} \Z/p^r, \Z/p^r \right) \,. \]
Then the universal coefficient isomorphism follows by taking cohomology and
observing that $\Z/p^r$ is an injective $\Z/p^r$-module.  All of these
isomorphisms are compatible with changing the level.
\end{proof}
\begin{proof}[Proof of Theorem \ref{homology vanishing}]
To prove the first claim, we use Lemma \ref{universal coefficient} to write
\begin{equation}\label{eq:precise dual}
 \varprojlim_{K \supseteq H} H^{\BM}_i(X_K,\Z/p^r) \simeq
\Hom_{\Z/p^r}\left(\varinjlim_{K \supseteq H} H^i_c(X_K,\Z/p^r),\Z/p^r \right) 
\end{equation}
and then apply Corollary \ref{strongest vanishing theorem}. 
From Theorem \ref{thm:ordinary PD} and Theorem \ref{thm:completed hom via distributions}, we obtain
the Poincar\'e duality spectral sequence
\[ \Ext^i_{\sD(T_0)}(\tH_j(N_1)^{T^+-\ord},\sD(T_0)) \Rightarrow \tH_{2d-i-j}^{\BM}(N_1)^{T^+-\ord} \,. \]
Then the second claim follows from the first claim and the above spectral sequence by the same
argument as in \cite[Corollary 4.2.3]{scholze-galois} (note that $D$ from Theorem~\ref{thm:ordinary PD} can be identified with $2d$).
\end{proof}
\bibliographystyle{amsalpha}
\bibliography{ArizonaProjectBib}

\end{document}